\documentclass[11pt, reqno]{amsart}
\usepackage{amsmath}
\usepackage{amsthm}
\usepackage{bbm}
\usepackage{nicefrac}
\usepackage{graphicx}
\usepackage{amssymb}
\usepackage{color}
\usepackage{verbatim}
\usepackage{cite}
\usepackage[font=small,skip=5pt]{caption}
\usepackage{enumitem}
\usepackage[makeroom]{cancel}
\usepackage{mathtools}

\usepackage{tikz}
\usepackage{pgfplots}
\pgfplotsset{compat=1.15}
\usepackage{mathrsfs}
\usetikzlibrary{arrows}

\usepackage{graphicx,calc}
\newlength\myheight
\newlength\mydepth
\settototalheight\myheight{Xygp}
\newtheorem{theorem}{Theorem}[section]
\newtheorem{lemma}[theorem]{Lemma}
\newtheorem{proposition}[theorem]{Proposition}
\newtheorem{corollary}[theorem]{Corollary}
\newtheorem{conjecture}[theorem]{Conjecture}

\theoremstyle{definition}
\newtheorem{definition}[theorem]{Definition}
\newtheorem{example}[theorem]{Example}
\usepackage{tikz}
\usepackage{stackrel}

\theoremstyle{remark}
\newtheorem{remark}[theorem]{Remark}

\newcommand{\CO}{\mathcal{O}}

\DeclareMathOperator{\mult}{mult}
\DeclareMathOperator{\Div}{div}
\DeclareMathOperator{\cone}{Cone}
\DeclareMathOperator{\conv}{Conv}

\DeclareMathOperator{\lcm}{lcm}

\title{Generation of jets and Fujita's jet ampleness conjecture on toric varieties} 
\author{Jos\'e Luis Gonz\'alez and Zhixian Zhu}
\address{Department of Mathematics, University of California, Riverside, Riverside, CA 92521}
\email{jose.gonzalez@ucr.edu}
\address{Capital Normal University, Academy for multidisciplinary studies, Haidian, Beijing, China, 100081}
\email{zhixian@cnu.edu.cn}       

\let\oldbibliography\thebibliography
\renewcommand{\thebibliography}[1]{\oldbibliography{#1}
\setlength{\itemsep}{5pt}} 

\begin{document}

\vspace{-8mm}

\maketitle

\vspace{-9.5mm}

\begin{abstract}
Jet ampleness of line bundles generalizes very ampleness by requiring 
the existence of enough global sections to separate not just points and tangent vectors, but also their higher order analogues called jets.
We give sharp bounds guaranteeing that a line bundle on a projective toric variety is $k$-jet ample 
in terms of its intersection numbers with the invariant curves,    
in terms of the lattice lengths of the edges of its polytope,   
in terms of the higher concavity of its piecewise linear function    
and in terms of its Seshadri constant.
For example, the tensor power $k+n-2$ of an ample line bundle on a projective toric variety of dimension $n \geq 2$ 
always generates all $k$-jets, but might not generate all $(k+1)$-jets.
As an application, we prove the $k$-jet generalizations of Fujita's conjectures on toric varieties with arbitrary singularities.
\end{abstract}


\vspace{-2.5mm}

\section{Introduction}

\vspace{-2.5mm}

In this article we study jet ampleness of line bundles on projective toric varieties with arbitrary singularities.
Jet ampleness was introduced by Beltrametti, Francia and Sommese in \cite{BFS89} as a property of those line bundles that intuitively yield higher order embeddings.
Jet ampleness requires that the line bundle has enough global sections to separate not just points and tangent vectors, but also their higher order analogues called jets.
More precisely, a line bundle $L$ on a projective variety $X$ is said to be $k$-jet ample, if the evaluation map 
\begin{equation}  \label{equation.evaluation.map.definition}
H^0(X, L)\rightarrow H^0(X, L\otimes \CO_X/(m_{x_1}^{k_1}\otimes m_{x_2}^{k_2}\otimes\cdots\otimes m_{x_r}^{k_r}))
\end{equation}
is surjective, for any set $\{x_1, x_2,\ldots, x_r\}$ of distinct closed points of $X$, and any positive integers $k_1,\ldots,k_r$ such that $\sum_{i=1}^{r}k_i=k+1$.
In this definition, $k$ can be any nonnegative integer
and the $m_{x_i}$ denote the ideal sheaves of the points $x_i$. 
For example, $0$-jet ampleness is equivalent to global generation and
$1$-jet ampleness is equivalent to very ampleness.
From the definition, one can see that if $L_1$ is a $k_1$-jet ample line bundle and $L_2$ is a $k_2$-jet ample line bundle then $L_1 \otimes L_2$ is $(k_1+k_2)$-jet ample (see \cite[Lemma 2.2]{BS93}).
In particular, the { tensor} product of $k$ very ample line bundles is always $k$-jet ample.
Furthermore, for any fixed $k$, a sufficiently high power of a given ample line bundle becomes $k$-jet ample.

In the case of smooth projective toric varieties, Di Rocco showed that a line bundle $L$ is $k$-jet ample if its intersection with each $T$-invariant curve is at least $k$ \cite[Proposition~3.5, Theorem~4.2]{DiRoccoJets99}.
However, for $k \geq 0$ and $\operatorname{dim}X \geq 2$, if we allow singularities on the projective toric variety $X$, 
we will see below that a line bundle $L$ is $k$-jet ample if its intersection with each $T$-invariant curve is at least $k+n-2$, and moreover this bound is sharp (see Section~\ref{subsection.theorems}).

Our main results provide sufficient conditions for the $k$-jet ampleness of a line bundle on a projective toric variety 
in terms of its intersection numbers with the invariant curves,    
in terms of the lattice lengths of the edges of its polytope,     
in terms of the higher concavity of its piecewise linear function     
{ and in terms of its Seshadri constant}  (see Section~\ref{subsection.theorems}). 
We will apply these results to prove the $k$-jet generalizations of Fujita's conjectures on projective toric varieties with arbitrary singularities (see Section~\ref{subsection.application}).

{ 
Like in our main results, Di Rocco's work in the smooth case also considered intersection numbers, lengths of edges, higher concavity and Seshadri constants. 
For our extension to the general case, we establish some properties of higher concavity and Seshadri constants on toric varieties which can be of independent interest.  
In Section~\ref{section.concavity}, we extend the notion of higher concavity of piecewise linear functions to the singular case.  
This property is closely related to intersection numbers and to the combinatorial properties of the associated polytope.
In Section~\ref{section.Seshadri}, we establish the lower semi-continuity in the Zariski topology of the Seshadri constant function on projective toric varieties and as a consequence one can compute the global Seshadri constant from the values at the $T$-invariant points. 
A key ingredient is the description of the Seshadri constant of an ample Cartier $T$-divisor at a point in \cite[Proposition 3.12 and Remark~3.14]{Ito12} (see also \cite[Corollary~4.2.2]{bauer2009primer}).
}

\subsection{Some results in the literature}

The problem of finding sufficient conditions for jet ampleness has been studied on several classes of varieties.
In \cite{BaSz97,BaSz97-Surfaces}, Bauer and Szemberg gave criteria for the $k$-jet ampleness of line bundles on abelian varieties.
For example, \cite[Theorem 1]{BaSz97} shows that if $L_1,\ldots,L_{k+2}$ are ample line bundles on an abelian variety then $L_1 \otimes \cdots \otimes L_{k+2}$ is $k$-jet ample.
Bauer, Di Rocco and Szemberg in \cite{BRSz00-Jets-K3} and Rams and Szemberg in \cite{RamsSz04-Jets-K3} give criteria for the $k$-jet ampleness of line bundles on smooth $K3$  surfaces.
For example, \cite[Theorem]{RamsSz04-Jets-K3} shows that if $L$ is an ample line bundle on a smooth $K3$ surface, then $L^{\otimes m}$ is $k$-jet ample for any $m \geq \max\{2k+1,2\}$.  
Moreover, \cite[Proposition 1.2]{BRSz00-Jets-K3} gives examples of ample line bundles $L$ on smooth $K3$ surfaces such that $L^{\otimes m}$ is not $k$-jet ample for $2k \geq m \geq 1$, and hence the bound in \cite[Theorem]{RamsSz04-Jets-K3} is sharp. 
%
In \cite{Calabi-Yau-Threefolds-2000}, Beltrametti and Szemberg gave criteria for the $k$-jet ampleness of line bundles on smooth Calabi-Yau threefolds.
For example, \cite[Theorem~1.1 (1)]{Calabi-Yau-Threefolds-2000} shows that if $L$ is an ample line bundle on a smooth Calabi-Yau threefold and $k \geq 2$, then $L^{\otimes m}$ is $k$-jet ample for any $m \geq 3k$.
In \cite{Chintapalli2014}, Chintapalli and Iyer gave criteria for the $k$-jet ampleness of line bundles on smooth hyperelliptic varieties.
For example, \cite[Theorem~1.2]{Chintapalli2014} shows that if $L$ is an ample line bundle on a smooth hyperelliptic variety, then $L^{\otimes m}$ is $k$-jet ample for any $m \geq k+2$.
In \cite{Farnik2016}, Farnik gave criteria for the $k$-jet ampleness of line bundles on hyperelliptic surfaces.
In \cite[Theorem~4.1]{Farnik2016} she showed that if $L$ is a line bundle of type $(m,m)$ on a smooth hyperelliptic surface, with $m \geq k+2$, then $L$ is $k$-jet ample.
We point out that in the literature one can also find other notions of higher order embeddings.
In addition to $k$-jet ampleness, Beltrametti, Francia and Sommese also introduced and studied the concepts of $k$-spannedness and $k$-very ampleness in \cite{BFS89,BS90,BS91,BS93}.
These notions are also widely studied and they are weaker than $k$-jet ampleness in general.

\begin{remark}[Remark on terminology]
For convenience, one says that a line bundle $L$ on a projective variety $X$ generates $k$-jets over a subset $Z$ of $X$, when the evaluation map in (\ref{equation.evaluation.map.definition}) is surjective for any set $\{x_1, x_2,\ldots, x_r\}$ of distinct closed points of $Z$, and any positive integers $k_1,\ldots,k_r$ such that $\sum_{i=1}^{r}k_i=k+1$.
Hence, $L$ is $k$-jet ample if and only if it generates $k$-jets over any finite subset of $X$. 
Likewise, a Cartier divisor $D$ on $X$ is said to be $k$-jet ample when the line bundle $\mathcal{O}_{X}(D)$ is $k$-jet ample and is said to generate $k$-jets over $Z$ when $\mathcal{O}_{X}(D)$ does.
\end{remark}

\subsection{Our main results on $k$-jet ampleness}     \label{subsection.theorems}


Our main result Theorem~\ref{theorem.jet.ampleness.gamma} is a sufficient condition for any ample Cartier divisor on a projective toric variety $X$ to be $k$-jet ample. 
This sufficient condition is written in terms of a nonnegative constant $\Gamma_X$ associated to the toric variety which can be computed directly (see Definition~\ref{definition.gamma.toric.variety}).
When the projective toric variety is additionally smooth, the constant $\Gamma_X$ is zero and Theorem~\ref{theorem.jet.ampleness.gamma} specializes to the main result on the generation of jets on smooth toric varieties in \cite[Theorem 4.2]{DiRoccoJets99}. 
We refer to Section~\ref{section.concavity} for the definition and basic properties of higher concavity of piecewise linear functions on fans. { We refer to Definition~\ref{definition.seshadri} to review the notion of the Seshadri constant.}

\begin{theorem}     \label{theorem.jet.ampleness.gamma}
Let $D$ be an ample Cartier $T$-divisor 
on a projective toric variety $X$.   
Suppose that for some $k \in \mathbb{Z}_{\geq 0}$, 
any of the following equivalent conditions holds,
\begin{enumerate}
    \item \label{thm.condition.intersection} $D \cdot C \geq k+\Gamma_X$ for each complete $T$-invariant curve $C$ in $X$,
    \item \label{thm.condition.length} Each edge of the polytope $P_D$ has lattice length at least $k+\Gamma_X$,
    \item \label{thm.condition.concavity} The piecewise linear function $\psi_D$ is $(k+\Gamma_X)$-concave,
    \item \label{thm.condition.seshadri} { The Seshadri constant $\varepsilon(D)$ is at least $k+\Gamma_X$,}
\end{enumerate}
then $D$ is $k$-jet ample. 
\end{theorem}



On a projective toric variety $X$ of dimension $n \geq 2$, by Lemma~\ref{lemma.gamma.cone} the constant $\Gamma_X$ always satisfies $0 \leq \Gamma_X \leq n-2$, and therefore we get the following friendlier version for the working mathematician. 

\begin{corollary}     \label{corollary.jet.ampleness.k+n-2}
Let $D$ be an ample Cartier $T$-divisor 
on a projective toric variety $X$ of dimension $n \geq 2$.   
Suppose that for some $k \in \mathbb{Z}_{\geq 0}$, 
any of the following equivalent conditions holds,
\begin{enumerate}
    \item \label{corollary.condition.intersection} $D \cdot C \geq k+n-2$ for each complete $T$-invariant curve $C$ in $X$,
    \item \label{corollary.condition.length} Each edge of the polytope $P_D$ has lattice length at least $k+n-2$,
    \item \label{corollary.condition.concavity} The piecewise linear function $\psi_D$ is $(k+n-2)$-concave,
    \item  \label{corollary.condition.seshadri} { The Seshadri constant $\varepsilon(D)$ is at least $k+n-2$,}
\end{enumerate}
then $D$ is $k$-jet ample. This bound is sharp. 
\end{corollary}


\begin{example}  \label{example.jet.ampleness.curves}
If $D$ is an ample Cartier divisor 
on a projective toric variety $X$ of dimension $n \geq 2$,    
such that for some $k{\kern-0.13em} \in{\kern-0.13em} \mathbb{Z}_{\geq 0}$,{\kern-0.1em} 
$D \cdot C{\kern-0.02em} \geq{\kern-0.02em} k + n - 2$ for each complete $T$-invariant curve $C$ in $X$,\!  
then $D$ is $k$-jet ample. 
Indeed, this follows by applying Corollary~\ref{corollary.jet.ampleness.k+n-2} to a $T$-divisor linearly equivalent to $D$. 
\end{example}

\begin{example}    \label{example.products}
If $D_1,\ldots,D_{k+n-2}$ are ample Cartier divisors on a projective toric variety of dimension $n \geq 2$, then $D=D_1 + \cdots + D_{k+n-2}$ is $k$-jet ample.
Indeed, this follows from Example~\ref{example.jet.ampleness.curves} since $D \cdot C = \sum_{i=1}^{k+n-2} D_i \cdot C \geq k+n-2$, for each complete $T$-invariant curve $C$ in $X$.
\end{example}


The following theorem answers the natural question of how high a multiple of an ample Cartier divisor on a projective toric variety one needs to take to guarantee $k$-jet ampleness.
This result is sharp if one considers all projective toric varieties of a given dimension.

\begin{theorem}   \label{theorem.jet.ampleness.multiple}
Let $D$ be an ample Cartier divisor on a projective toric variety $X$ of dimension $n \geq 2$.
Then, $mD$ is $k$-jet ample for each $m \geq k+n-2$ and each $k \in \mathbb{Z}_{\geq 0}$.
This bound is sharp.
\end{theorem}



Our main result Theorem~\ref{theorem.jet.ampleness.gamma} will follow from the following stronger, but more technical sufficient condition for $k$-jet ampleness, with assumptions about each maximal cone rather than about the collection of maximal cones as a whole. For each maximal cone $\sigma$, we will define an integer valued invariant $L_\sigma$ (see Definition~\ref{defition.minimum.length}), to measure the positivity of $D$ and a rational number $\Gamma_{\sigma^\vee}$ (see Definition~\ref{definition.gamma}), to measure the difficulty of separating $k$-jets at the invariant point $x_\sigma$ of $X$. 

\begin{theorem}    \label{theorem.jet.ampleness.technical}
Let $D$ be an ample Cartier $T$-divisor 
on a projective toric variety $X=X(\Delta)$.   
If for some $k \in \mathbb{Z}_{\geq 0}$, 
$L_{\sigma} \geq k + \Gamma_{\sigma^{\vee}}$, for each maximal cone $\sigma \in \Delta$,  
then $D$ is $k$-jet ample. 
\end{theorem}

\subsection{Application: Fujita{'}s k-jet ampleness conjecture} \label{subsection.application}

A basic question about adjoint linear systems is when they are base point free or very ample.  In \cite{Fuj88}, Fujita raised the following conjectures. 
\begin{conjecture}
Let $X$ be an $n$-dimensional projective algebraic variety, smooth or with mild singularities, and let $D$ be an ample divisor on $X$.
\begin{itemize}
    \item[(i)](freeness) If $\ell \geq n+1$, then $\ell D+K_{X}$ is globally generated.
    \item[(ii)](very ampleness) If $\ell \geq n+2$, then $\ell D+K_{X}$ is very ample.
\end{itemize}
\end{conjecture}

For smooth curves, both conjectures follow from the Riemann-Roch formula. 
In characteristic $0$, Fujita's freeness conjecture has been proved up to dimension $5$, while Fujita's very ampleness conjecture is only proved for surfaces. 
In dimension $2$, Reider proved both of Fujita's conjectures using the Bogomolov inequality. 
In dimension $3$, the freeness conjecture was proved using multiplier ideal sheaves by Ein and Lazarsfeld in \cite{EL93}.  
In dimension $4$, the freeness conjecture was proved by Kawamata in \cite{Kaw97}
and in dimension $5$ it was proved by Ye and Zhu in \cite{YZ15}.    


In \cite{Mustata02}, Musta\c{t}\u{a} gave a characteristic-free approach to vanishing results on arbitrary toric varieties. As an application, in \cite[Theorem 0.3]{Mustata02} he proved Fujita's freeness and very ampleness conjectures on smooth toric varieties of arbitrary characteristic. 
For singular toric varieties, Fujino proved a generalization of Fujita's freeness conjecture \cite{Fuj03}. Fujita's very ampleness conjecture was proved by Payne in \cite{Pay06}. In fact, they both showed that a smaller intersection number suffices, unless $X$ is $\mathbb{P}^n$ and $D$ is linearly equivalent to a hyperplane. 

On the other hand, one may hope that for a smooth projective variety $X$ and an ample Cartier divisor $D$, $K_X+(n+k+1)D$ generates $k$-jets if $D$ is ample. However, this fails even for the generation of jets at one point, see \cite{EKL95}.     
In Section~\ref{section.application}, we prove that a stronger generalization of Fujita's conjectures for $k$-jets still holds for projective toric varieties with arbitrary singularities. 

\begin{theorem}\label{thm.Fujita}
Let $X$ be a projective $n$-dimensional toric variety not isomorphic to $\mathbb{P}^n$. Let $D$ and $D'$ be $T$-invariant $\mathbb{Q}$-Cartier divisors such that $0\ge D'\ge K_X$ and $D+D'$ is Cartier. Let $k$ be a non-negative integer. If $D\cdot C\ge n+k$ for every $T$-invariant curve $C$, then $D+D'$ is $k$-jet ample. 
\end{theorem}

\subsection{This article is organized as follows}
In Section~\ref{section.proofs}, we establish sufficient conditions for jet ampleness on {projective} toric varieties with arbitrary singularities.
In Section~\ref{section.examples}, we present examples illustrating our conditions and their sharpness.
In Section~\ref{section.application}, we prove $k$-jet generalizations of Fujita's conjectures on {projective} toric varieties with arbitrary singularities.
In Section~\ref{section.concavity}, we introduce a notion of higher concavity for piecewise linear functions on fans and study its basic properties. 
{ Lastly, in Section \ref{section.Seshadri}, we prove that the Seshadri constant function on projective toric varieties  is lower semi-continuous in the Zariski topology and its global value can be computed using the invariant points.}

\subsection{Notation and conventions}   \label{preliminaries.notation}

%
All of our varieties are defined over a fixed algebraically closed field $\mathbbm{k}$ of arbitrary characteristic.    
Throughout, $N$ denotes a lattice and $M$ denotes its dual lattice $\operatorname{Hom}(N,\mathbb{Z})$. 
We set $N_{\mathbb{R}}=N \otimes \mathbb{R}$ and $M_{\mathbb{R}}= M  \otimes \mathbb{R}$.
We denote the associated pairing by $\langle u, v \rangle $, for any elements $u \in M_{\mathbb{R}}$ and $v\in N_{\mathbb{R}}$.
We denote by $T$ the split torus with one-parameter subgroup lattice $N$ and character lattice $M$.
Throughout, $\Delta$ denotes a fan in $N_{\mathbb{R}}$ and $\operatorname{Supp}(\Delta) \subseteq N_{\mathbb{R}}$ denotes its support.
By facet of a cone we mean a codimension-one face.    
We say that an element of a lattice is primitive when it is nonzero and it is not a positive integer multiple of a lattice element, other than itself.
The lattice length of the linear segment between $u_1,u_2 \in M \otimes \mathbb{Q} \subseteq M \otimes \mathbb{R}$ is the unique nonnegative rational number $l$ such that $ u_1-u_2=lw$, for some $w \in M= M \otimes \mathbb{Z} \subseteq M \otimes \mathbb{Q}$ primitive.
For each cone $\sigma \in \Delta$, the lattice associated to $\sigma$, denoted by $N_\sigma$, is the lattice obtained by intersecting $N$ with the linear subspace of $N_{\mathbb{R}}$ spanned by $\sigma$.
We denote the semigroup algebra associated to a semigroup $S$ by $\mathbbm{k}[S]=\mathbbm{k}[\, \chi^u \, | \, u \in S ]$.

The affine toric variety associated to a cone $\sigma \in \Delta$ is denoted by $U_{\sigma}$ and the toric variety associated to the fan $\Delta$ is denoted by $X(\Delta)$.
On any toric variety $X=X(\Delta)$, we will denote by $K_X$ the negative of the sum of the $T$-invariant prime divisors, which is a canonical divisor on $X$. 
The $T$-invariant point of $X(\Delta)$ associated to a maximal cone $\sigma \in \Delta$ is denoted by $x_{\sigma}$.  
Given a $T$-divisor $D$ on $X(\Delta)$, we denote its associated polyhedron by $P_D \subseteq M_{\mathbb{R}}$.
For each vertex $u$ of $P_D$, we denote the cone in $M_{\mathbb{R}}$ over the translated polyhedron $P_D-u$ by $\cone(P_D,u)$. 
Recall that the points in $P_D \cap M$ are in correspondence with a basis of $H^0(X(\Delta),D)$ and that $P_D$ is a polytope if $X$ is complete.  
If $D$ is a Cartier $T$-divisor on $X(\Delta)$,
we denote its associated piecewise linear function by $\psi_D: \operatorname{Supp}(\Delta) \rightarrow \mathbb{R}$.
Recall that if for a cone $\sigma \in \Delta$, we choose $u_\sigma \in M$ 
such that $\operatorname{div}(\chi^{-u_\sigma})|_{U_{\sigma}}=D|_{U_{\sigma}}$, then 
$\psi_D(v)=\langle u_{\sigma},v \rangle$, for each $v \in \sigma$. 

Given a convex subset $S$ of a real vector space, we will follow the convention that a real valued function $f:S \rightarrow \mathbb{R}$ is said to be concave if $f(tx+(1-t)y) \geq tf(x)+(1-t)f(y)$, for all $x,y \in S$ and $t \in [0,1]$.
Given vectors $v,w,w_1,\ldots,w_n$ in a real vector space, 
we denote the ray from $v$ to $w$ by $\overrightarrow{vw}$;   
the line segment joining $v$ and $w$ by $\overline{vw}$;  
the linear span of $w_1,\ldots,w_n$ by $\langle w_1,\ldots,w_n \rangle_{\mathbb{R}}$; 
the convex hull of $w_1,\ldots,w_n$ by $\conv(w_1,\ldots,w_n)$; 
the cone spanned by $w_1,\ldots,w_n$ by $\cone(w_1,\ldots,w_n)$;   
and the translated cone $w+\cone(w_1-w,\ldots,w_n-w)$ by $\cone_w(w_1,\ldots,w_n)$. 
By a polytope we mean the convex hull of a finite set of points.   
Unless specified otherwise by the context, by a cone we mean a strictly convex, rational, polyhedral cone.

\subsection*{Acknowledgements} 
The authors would like to thank {Javier Gonz\'alez-Anaya}, {Atsushi Ito}, Mircea Musta\c{t}\u{a} and Lei Song for insightful conversations. Zhixian Zhu would also like to thank Tong Zhang for inviting her to visit East China Normal University, where part of this work was completed.  
Jos\'e Gonz\'alez was supported by a grant from the Simons Foundation (Award Number 710443) and by the UCR Academic Senate. 


\section{Generation of jets on toric varieties}    \label{section.proofs}

In this section, we prove our main results on jet ampleness on projective toric varieties stated in Section~\ref{subsection.theorems}.

\subsection{The maximum weight function $W^Q_{\max}$ and the constants $\Gamma_Q$ and $\Gamma_X$}

\begin{definition} \label{definition.weight.function}
Let $Q$ be a cone in a lattice $M$ and let $w_1,\ldots,w_{m}$ be the primitive lattice generators of its rays. 
We define the maximum weight function $W_{\max}^Q: Q \rightarrow \mathbb{R}$ by 
\[W_{\max}^Q(u) := \max\left\{ \left. \sum_{i=1}^{m}a_{i} \,  \right| \, u= \sum_{i=1}^{m}a_{i}w_{i} \textnormal{ and } a_i \geq 0 \textnormal{ for all $i$} \right\}.
\]
for any $u \in Q$. When $Q$ is simplicial we will sometimes write simply $W^Q$ to denote $W_{\max}^Q$.
\end{definition}

\begin{remark}     \label{remark.piecewise.linear}
Since all our cones are rational and strictly convex the function $W_{\max}^Q$ in Definition~\ref{definition.weight.function} is well-defined.
Moreover, it is well-known that $W_{\max}^Q$ is piecewise linear on $Q$, that is, there exists a subdivision of $Q$ into finitely many subcones such that the restriction of $W_{\max}^Q$ to each of these subcones is linear.  
One way to construct such a subdivision without introducing additional rays is to consider the cones over the faces of the polytope $\conv(w_1,\ldots,w_{m})$ that are visible from the origin, see \cite[Section 3]{Pay06} for details.
\end{remark}


\begin{definition}   \label{definition.gamma}
Let $Q$ be a cone in a lattice $M$ and let $w_1,\ldots,w_{m}$ be the primitive lattice generators of its rays. 
Let $m_Q$ denote the irrelevant maximal ideal of $\mathbbm{k}[Q \cap M]$.
For each $u \in Q \cap M$, let $k^Q_u=\max\{k \in \mathbb{Z}_{\geq 0} \ | \ \chi^u \in m_Q^k \} $.
Let $S_Q=\{ \sum_{i=1}^{m}a_{i}w_i \, | \, 0 \leq a_i < 1 \textnormal{ for all $i$} \}$.
Notice that $S_Q \cap M$ is a finite set. 
When there is no risk of ambiguity we will simply write $k_u$ and $S$ to denote $k_u^Q$ and $S_Q$, respectively.
We define the constant $\Gamma_Q$ associated to $Q$ as 
\[
\Gamma_Q := \max \{ W^Q_{\max}(u) - k^Q_u \, | \,  u \in S_Q \cap M  \}. 
\]
\end{definition}

\begin{remark}
In Definition~\ref{definition.gamma}, we see that $W^Q_{\max}(0) - k_0$=0 and then $\Gamma_Q \geq 0$ for any cone $Q$.
Moreover, notice that for every $u \in Q \cap M$, there exists $w \in S_Q \cap M$ such that $W^Q_{\max}(w) - k_w \geq W^Q_{\max}(u) - k_u$.
Hence, $\Gamma_Q = \max \{ W^Q_{\max}(u) - k_u \ | \  u \in Q \cap M  \}$.
\end{remark}


\begin{remark}
Given a cone $Q$ of dimension $n \geq 2$, the maximum weight $W^{Q}_{\max}(u)$ of any  irreducible element $u$ in the semigroup $Q \cap M$ is at most $n-1$.
Using Remark~\ref{remark.piecewise.linear}, the proof of this claim reduces to the case where $Q$ is simplicial, which is proved in \cite[Theorem 2]{EW91}.  
\end{remark}

We will now use \cite[Theorem 2]{EW91} to get bounds for $\Gamma_Q$.

\begin{lemma}     \label{lemma.gamma.cone}
For any smooth cone $Q$, we have $\Gamma_Q=0$.   
In particular, for any cone $Q$ of dimension $n \leq 1$, we have $\Gamma_Q=0$.
For any cone $Q$ of dimension $n \geq 2$, we have $0 \leq \Gamma_Q \leq n-2$. 
\end{lemma}
\begin{proof}
If $Q$ is smooth then $S_Q \cap M =\{ 0 \}$, and then $\Gamma_Q=0$.
Assume now that $Q$ has dimension $n \geq 2$.
Let $u \in Q \cap M$ be such that $\Gamma_Q=W^Q_{\max}(u)-k^Q_u$.
By Remark~\ref{remark.piecewise.linear}, we can choose a subcone $Q'$ of $Q$ whose rays are rays of $Q$, such that $u \in Q'$ and $W^Q_{\max}$ is linear on $Q'$.
It follows that $W^Q_{\max}|_{Q'}=W^{Q'}_{\max}$ and in particular $W^Q_{\max}(u)=W^{Q'}_{\max}(u)$. %
We can now choose a simplicial cone $Q''$ whose rays are rays of $Q'$ such that $u \in Q''$.
We have that $W^Q_{\max}(u)=W^{Q'}_{\max}(u)=W^{Q''}(u)$ and $k^{Q}_{u} \geq k^{Q'}_{u} \geq k^{Q''}_{u}$.
Hence, $W^Q_{\max}(u)-k^Q_u \leq W^{Q''}(u)-k^{Q''}_u$.
Therefore we can assume that the given cone $Q$ was simplicial of dimension $n \geq 2$. 
By contradiction, let us suppose that $\Gamma_Q=W^Q(u)-k^Q_u > n-2$.
By appropriately replacing $u$ if necessary, we can assume that $u \in S_Q \cap M$. 
Notice that $u\neq 0$, then $\chi^u \in m_Q$, and hence $W^Q(u) > n- 2 + k^Q_u \geq n-1$.
Since $W^Q(u) > n-1$, then $\chi^u \in m_Q^2$ by \cite[Theorem 2]{EW91}.
Then, $W^Q(u) > n-2 + k^Q_u \geq n$, but this is a contradiction since $W^Q(w) \leq n$ for each $w \in S_Q$. 
\end{proof}


We can use the constant $\Gamma_Q$ to get the following test for ideal membership.

\begin{lemma}\label{lemma.generators}
Let $Q$ be a cone in a lattice $M$. 
Let $R=\mathbbm{k}[Q \cap M]$ and let $m_Q$ be its irrelevant ideal.
Then, for each $k \in \mathbb{Z}_{\geq 0}$ we have
\[
\left\{ \chi^u \in R \, \left. \right| \, W^{Q}_{\max}(u) > k-1+\Gamma_Q \right\} \subseteq m_Q^{k}.
\]
In particular, $R/m_Q^k$ is spanned over $\mathbbm{k}$ by the images of 
$\{ \chi^u \in R \, | \, W^{Q}_{\max}(u) \leq k-1+\Gamma_Q \}$.
\end{lemma}

\begin{proof}
We will use the notation from Definition~\ref{definition.gamma}.  
Suppose that $W^{Q}_{\max}(u) > k-1+\Gamma_Q$, for some $u \in Q \cap M$.
Then, there exists an expression $u=\sum_{i=1}^{m}(a_{i}+\alpha_i)w_i$, 
with $a_i \in \mathbb{Z}_{\geq 0}$ and $0 \leq \alpha_i < 1$, for each $1 \leq i \leq m$, 
such that $\sum_{i=1}^{m}a_{i} + \sum_{i=1}^{m} \alpha_i > k-1+\Gamma_Q$.
Since $w:=\sum_{i=1}^{m}\alpha_iw_i=u-\sum_{i=1}^{m}a_{i}w_i$ is in $S_Q \cap M$, then $\Gamma_Q \geq W^{Q}_{\max}(w) - k_w \geq \sum_{i=1}^{m}\alpha_i - k_w$.
Then,
\[
\sum_{i=1}^{m}a_{i} + \sum_{i=1}^{m} \alpha_i > k-1 + \Gamma_Q \geq  k-1 + \sum_{i=1}^{m}\alpha_i - k_w.
\]
Hence $\sum_{i=1}^{m}a_{i} +k_w > k-1$, and thus $\sum_{i=1}^{m}a_{i} +k_w \geq k$.
We know that $\chi^w \in m_Q^{k_w}$, and therefore, 
\[
\chi^u = \chi^{\sum_{i=1}^m a_iw_i}\chi^w \in m_Q^{\sum_{i=1}^m a_i}m_Q^{k_w}=m_Q^{\sum_{i=1}^m a_i+k_w}\subseteq m_Q^k. \qedhere
\]  
\end{proof}


\begin{definition}       \label{definition.gamma.toric.variety}
For any toric variety $X$ associated to a fan $\Delta$ whose maximal cones are top-dimensional,  
we define  
$\Gamma_X := \max \{ \Gamma_Q \, | \,  Q=\sigma^{\vee} \textnormal{ for some maximal cone } \sigma \in \Delta  \}$.   
\end{definition}

\begin{remark}
It follows from Lemma~\ref{lemma.gamma.cone} that for a toric variety $X$ associated to a fan $\Delta$  whose maximal cones are top-dimensional,  
we have that $\Gamma_X=0$ if $X$ is smooth (for example, if $X$ has dimension $n=1$) and that $0 \leq \Gamma_X \leq n-2$ if $X$ has dimension $n \geq 2$.   
\end{remark}


\subsection{Reduction of generation of jets to the $T$-invariant case}

In this subsection, we will see that the generation of jets on complete toric varieties can be reduced to the case of jets supported at the $T$-invariant points. We start by recalling Borel fixed-point theorem.

\begin{theorem}[Borel Fixed-Point Theorem {\cite[\S 21.2]{Linear.Algebraic.Groups.Humphreys}}] 
Let $G$ be a connected solvable algebraic group, and let $X$ be a
nonempty complete variety on which $G$ acts regularly. Then $G$ has a fixed-point in $X$.
\end{theorem}

The argument for the reduction of the generation of jets to the $T$-invariant case in \cite[Page 180]{DiRoccoJets99} for smooth complete toric varieties carries to the not necessarily smooth case. We recall  this argument in the next proposition.

\begin{proposition}    \label{proposition.invariant.enough}
Let $D$ be a Cartier $T$-divisor on a complete toric variety $X$. If the divisor $D$ generates $k$-jets on any subset $Z$ of $X$ supported at $T$-invariant points, then $D$ is $k$-jet ample. 
\end{proposition}
\begin{proof}
Let us fix a non-negative integer $k$. 
We define a subset $Y$ of the $(k+1)$-fold product $X^{k+1}$ as follows. 
A closed point $y=(y_1,\ldots,y_{k+1}) \in X^{k+1}$ is in $Y$ if and only if the evaluation map \[
H^0(X, \CO_X(D) ) \rightarrow H^0(X, \CO_X(D) \otimes \CO_X/(m_{x_1}^{k_1}\otimes m_{x_2}^{k_2}\otimes\cdots\otimes m_{x_r}^{k_r}))\]
is { not} surjective, where $x_1,\ldots,x_r \in X$ are the distinct points that occur as components of $y$ and $k_1,\ldots,k_r \in \mathbb{Z}^+$ are the respective number of times that each of the $x_i$ occurs as a component of $y$.
Notice that the surjectivity of this evaluation map does not depend on the order we give to $x_1,\ldots,x_r$.  
Notice also that $\sum_{i=1}^{r}k_i=k+1$.  
By the nature of its definition, $Y$ is a closed subset of $X^{k+1}$, and in particular $Y$ is complete.
Moreover, since $D$ is a $T$-divisor then $Y$ is a $T$-invariant closed subset of $X^{k+1}$, where $X^{k+1}$ is considered with the induced componentwise action of $T$.
Since $T$ is connected, we get an induced action of $T$ on each irreducible component of $Y$.
By contradiction, let us assume that $D$ is not $k$-jet ample but generates $k$-jets on any $Z$ supported at $T$-invariant points.
Then, $Y$ is not empty. 
By applying Borel Fixed-Point Theorem to any of the irreducible components of $Y$, we deduce that $Y$ has a $T$-invariant closed point $y$.   
The components of this closed point $y$ are $T$-invariant closed points in $X$ and give us an instance of a non-surjective evaluation map for a set supported at the $T$-invariant points of $X$, which is a contradiction.
\end{proof}


\subsection{Generation of jets at one $T$-invariant point}     \label{subsection.one.point}

In this subsection, we will present a sufficient condition for the generation of jets supported at one $T$-invariant point, for ample Cartier $T$-divisors on projective toric varieties. We start by describing our strategy in the following remark.

\begin{remark}  \label{remark.maximal.ideal.powers}
Let $D$ be a Cartier $T$-divisor on a toric variety $X=X(\Delta)$.
We are interested in studying the surjectivity of the evaluation map 
\begin{equation} \label{equation.restriction}
H^0(X, \mathcal{O}_X(D))\rightarrow H^0(X, \mathcal{O}_X(D)\otimes \CO_X/m_{x}^{k+1})    
\end{equation}
for $k \in \mathbb{Z}_{\geq 0}$, when $x \in X$ is a $T$-invariant point.
Let $\sigma \in \Delta$ be the maximal cone corresponding to $x$ and let $u_\sigma \in M$ be such that $\operatorname{div}(\chi^{-u_{\sigma}})|_{U_\sigma}=D|_{U_\sigma}$.
Let $R=\mathbbm{k}[\cone(P_D,u_{\sigma}) \cap M]$ and let $m_R$ be its irrelevant maximal ideal. 
The evaluation map in (\ref{equation.restriction}) can be written as the composition, 
\[
H^0(X, \mathcal{O}_X(D))\xhookrightarrow{\xi} H^0(U_\sigma, \mathcal{O}_X(D))\xrightarrow{\psi} H^0(U_\sigma, \mathcal{O}_X(D)\otimes \CO_X/m_{x}^{k+1})=H^0(X, \mathcal{O}_X(D)\otimes \CO_X/m_{x}^{k+1}),
\]
where the restriction map $\xi$ is injective since $\mathcal{O}_X(D)$ is an invertible sheaf, the evaluation $\psi$ is surjective since $U_\sigma$ is affine, and the two spaces on the right are equal since $\CO_X/m_{x}^{k+1}$ is supported at the point $x \in U_\sigma \subseteq X$.
We have that $H^0(U_\sigma,\mathcal{O}_X(D))=\chi^{u_\sigma}\mathbbm{k}[\cone(P_D,u_{\sigma}) \cap M]$. 
Notice that a section $\chi^u$, such that $\chi^{u-u_\sigma} \in  m_{R}^{k+1}$, maps to zero under the evaluation $\psi$.
Therefore the space $H^0(X, \mathcal{O}_X(D)\otimes \CO_X/m_{x}^{k+1})$ is spanned over $\mathbbm{k}$ by the images of the sections $\chi^u$ such that $\chi^{u-u_\sigma} \in 
\mathbbm{k}[\cone(P_D,u_{\sigma})] 
\smallsetminus m_{R}^{k+1}
= R \smallsetminus m_{R}^{k+1}$.
Our task will be to show that under the right assumptions such sections can be lifted to $H^0(X, \mathcal{O}_X(D))$, and more generally do this when we are given
a collection of $T$-invariant points $x_1\ldots,x_r$ and we consider evaluation maps of the form
\[H^0(X, \mathcal{O}_X(D))\rightarrow H^0(X, \mathcal{O}_X(D)\otimes \CO_X/(m_{x_1}^{k_1}\otimes m_{x_2}^{k_2}\otimes\cdots\otimes m_{x_r}^{k_r})).\]
\end{remark}


\begin{definition}    \label{defition.minimum.length}
Let $D$ be an ample Cartier $T$-divisor on a projective toric variety $X=X(\Delta)$.
Notice that in this case there is a correspondence between maximal cones $\sigma$ in the fan $\Delta$ and vertices $u_\sigma$ of the polytope $P_D$.
For each maximal cone $\sigma \in \Delta$, we will denote by $L_{\sigma}^D$ the minimum of the lattice lengths of the edges of $P_D$ having $u_{\sigma}$ as a vertex.
When there is no risk of ambiguity, we will write simply $L_{\sigma}$ to denote $L_{\sigma}^D$.
\end{definition}


We now present a sufficient condition for the generation of jets at one $T$-invariant point.

\begin{proposition}\label{proposition.onept}
(Generation of jets at one $T$-invariant point)
Let $D$ be an ample Cartier $T$-divisor 
on a projective toric variety $X=X(\Delta)$.   
Let $x \in X$ be a $T$-invariant point and let $\sigma \in \Delta$ be its corresponding maximal cone.  
If $L_{\sigma} \geq k + \Gamma_{\sigma^{\vee}}$, for some $k \in \mathbb{Z}_{\geq 0}$, then the evaluation map
\[H^0(X, \mathcal{O}_X(D))\rightarrow H^0(X, \mathcal{O}_X(D)\otimes \CO_X/m_{x}^{k+1})\]
 is surjective.  
\end{proposition}
\begin{proof}
Note that $\mathcal{O}_X(D)$ is globally generated and 
$\{\chi^{u}\,|\,u\in P_D\cap M\}$ is a basis of $H^0(X, \mathcal{O}_X(D))$.
For $k=0$ the conclusion follows from the global generation of $\mathcal{O}_X(D)$, thus we assume that $k \geq 1$.
In particular, $L_\sigma \geq 1$. 
Let $u_\sigma$ be the vertex of the polytope $P_D$ corresponding to $\sigma$.  
Let $R=\mathbbm{k}[\cone(P_D,u_{\sigma}) \cap M]$ and let $m_R$ be its irrelevant maximal ideal. 
By Remark~\ref{remark.maximal.ideal.powers}, the space $H^0(X, \mathcal{O}_X(D)\otimes \CO_X/m_{x}^{k+1})$ is spanned over $\mathbbm{k}$ by the images of the sections $\chi^u$ such that $\chi^{u-u_\sigma} \in R \smallsetminus m_{R}^{k+1}$.
To show the surjectivity of the evaluation map, it is enough to show that $u \in P_D$ for all such sections.
For each such $\chi^u$, since $\chi^{u-u_\sigma} \notin m_R^{k+1}$, 
then by Lemma~\ref{lemma.generators} we have 
\begin{equation} \label{equation.inequality.weight.one.point}
W_{\max}^{\cone(P_D,u_{\sigma})}(u-u_\sigma) \le (k+1)-1+\Gamma_{\sigma^{\vee}} = k+\Gamma_{\sigma^{\vee}}\leq L_{\sigma}.
\end{equation}
Let $w_1,\ldots, w_{m}$ be the primitive generators of the rays of $\sigma^{\vee}$. 
By (\ref{equation.inequality.weight.one.point}),  
$u-u_\sigma = \alpha_1 w_1 +\ldots+ \alpha_m w_{m}$,
for some $\alpha_i \geq 0$ such that $\sum_{i=1}^{m} \alpha_i \leq L_{\sigma}$.
Therefore, we get the convex combination
\[
u = \frac{\alpha_1}{L_{\sigma}} (L_{\sigma}w_1+u_\sigma) +\cdots+ \frac{\alpha_m}{L_{\sigma}} (L_{\sigma}w_{m}+u_\sigma) + (1-\sum_{i=1}^{m}\frac{\alpha_i}{L_{\sigma}})u_\sigma.
\]
Notice that $w_1+u_\sigma,\ldots, w_{m}+u_\sigma$ are the first lattice points after $u_\sigma$ along the edges of $P_D$ containing $u_\sigma$ as a vertex. By definition, each such edge has lattice length at least $L_{\sigma}$.
Then, $L_{\sigma}w_1+u_\sigma,\ldots, L_{\sigma}{w_{m}}+u_\sigma$ are in $P_D$, and thus by convexity $u$ is also in $P_D$, as desired.
\end{proof}


\subsection{Generation of jets along two or more $T$-invariant points}     \label{subsection.two.or.more.points}

In this subsection, we will present a sufficient condition for the generation of jets supported along two or more $T$-invariant points, for ample Cartier $T$-divisors on projective toric varieties. 

\begin{lemma}    \label{lemma.non.adjacent}
Let $P$ be an $n$-dimensional polytope, $v$ be a vertex of $P$ and $v_1,\ldots,v_n$ be distinct vertices of $P$ each connected to $v$ by an edge.
Then, for any $u \in \conv(v,v_1,\ldots,v_n)$ and any $w \notin \{v,v_1,\ldots,v_n \} $ vertex of $P$, 
there exist $w_1,\ldots,w_{m}$ distinct vertices of $P$, each connected to $w$ by an edge,  
such that 
$u \in \cone_w(w_1,\ldots,w_{m})$, 
$\cone_w(w_1,\ldots,w_{m})$ is simplicial, 
and $u \notin \conv(w,w_1,\ldots,w_{m}) \smallsetminus \conv(w_1,\ldots,w_{m})$.
\end{lemma}
\begin{proof}
We may assume that the ambient space is $\mathbb{R}^n$ and that $w$ is the origin.
Let $w_1,\ldots,w_k$ be the distinct vertices of $P$ that are connected to $w$ by an edge.
Let $W:\cone(w_1,\ldots,w_k) \rightarrow \mathbb{R}$ be maximum weight function defined by 
\[W(x) := \max\left\{ \left. \sum_{i=1}^{k}a_{i} \,  \right| \, x = \sum_{i=1}^{k}a_{i}w_i \textnormal{ and } a_i \geq 0 \textnormal{ for all $i$} \right\}.
\]
Clearly $W(\lambda x)=\lambda W( x)$ for any $\lambda \geq 0$. 
Let $Q=\conv(w_1,\ldots,w_k)$.
Notice that for any $t >0$, the set $tQ$ consists precisely of the points $x$ in $\cone(w_1,\ldots,w_k)$ that can be written as $x=\sum_{i=1}^{k}a_{i}w_i$ with $a_i \geq 0$ for all $i$ and $\sum_{i=1}^{k}a_{i}=t$. 
Let us refer to the faces of $Q$ visible from $w$ as the lower faces of $Q$.
The lower faces of $Q$ consist precisely of the points that are not in $tQ$ for any $t>1$.
Then, the function $W$ is identically 1 on the lower faces of $Q$.
It follows that $W$ is a piecewise linear function on $\conv(w_1,\ldots,w_k)$ which is linear precisely on the cones over the lower faces of $Q$.

We claim that $W(z) \geq 1$ for any vertex $z$ of $P$ distinct from $w$.
By contradiction, assume that $W(z) < 1$. 
Let $z_0$ be the unique point in the ray $\overrightarrow{wz}$ such that $W(z_0)=1$. 
Then, $z_0$ is in the lower faces of $Q$. In particular $z_0 \in  Q$, and hence $z_0 \in P$.
Since the line segment $\overline{wz_0}$ is contained in $P$ and $z$ is in its relative interior, we deduce that $z$ cannot be a vertex of $P$, which is a contradiction.

Clearly $W(x+y) \geq W(x)+W(y)$, for any $x,y \in \cone(w_1,\ldots,w_k)$.
Since $u \in \conv(v,v_1,\ldots,v_n)$, and $W(v)\geq 1$ and $W(v_i)\geq 1$ for each $i$, it follows that $W(u) \geq 1$.
Let $u_0$ be the unique point in the ray $\overrightarrow{wu}$ such that $W(u_0)=1$.
Let $F$ be the unique face of $Q$ containing $u_0$ in its relative interior.
Notice that $F$ is a lower face of $Q$. 
The vertices of $Q$ are precisely $w_1,\ldots,w_k$. 
Hence, by relabeling we may assume the vertices of $F$ are $w_1,\ldots,w_r$, for some $r \leq k$.
Let $m$ be the positive integer such that the dimension of $F$ is $m-1$.
After relabeling if necessary, we may assume that $w_1,\ldots,w_{m}$ are such that $\conv(w_1,\ldots,w_{m})$ is $m-1$ dimensional and contains $u_0$.
We know that the function $W$ is linear on the cone over $F$ and hence also on $\cone(w_1,\ldots,w_{m})$.

By construction $\cone_w(w_1,\ldots,w_{m})=\cone(w_1,\ldots,w_{m})$ is simplicial and it contains $u$ since $u_0 \in \cone(w_1,\ldots,w_{m})$.
Moreover, since $W$ takes values strictly less than $1$ on $\conv(w,w_1,\ldots,w_{m}) \smallsetminus \conv(w_1,\ldots,w_{m})$,
then $u \notin \conv(w,w_1,\ldots,w_{m}) \smallsetminus \conv(w_1,\ldots,w_{m})$,  
as desired. 
\end{proof}


We now present a condition for the generation of jets along two or more $T$-invariant points.

\begin{proposition}\label{proposition.points}
(Generation of jets along two or more $T$-invariant points)
Let $D$ be an ample Cartier $T$-divisor 
on a projective toric variety $X=X(\Delta)$.   
Let $x_1,\ldots,x_r \in X$ be distinct $T$-invariant points 
and let $\sigma_1,\ldots,\sigma_r \in \Delta$ be their corresponding maximal cones, where $r \geq 2$. 
If for some $k \in \mathbb{Z}_{\geq 0}$, 
$L_{\sigma_i} \geq k + \Gamma_{\sigma_i^{\vee}}$, for each $1 \leq i \leq r$,  then for any positive integers $k_1,\ldots,k_r$ such that $\sum_{i=1}^{r} k_i=k+1$, the evaluation map 
\[H^0(X, \mathcal{O}_X(D))\rightarrow H^0(X, \mathcal{O}_X(D)\otimes \CO_X/(m_{x_1}^{k_1}\otimes m_{x_2}^{k_2}\otimes\cdots\otimes m_{x_r}^{k_r}))\]
is surjective. 
\end{proposition}
\begin{proof}
We let $n$ denote the dimension of $X$.    
For each $1 \leq i \leq r$, let $u_{i}$ be the vertex of the polytope $P_D$ corresponding to $\sigma_i$. 
We have an isomorphism
\[
H^0(X, \mathcal{O}_X(D)\otimes \CO_X/(m_{x_1}^{k_1}\otimes m_{x_2}^{k_2}\otimes\cdots\otimes m_{x_r}^{k_r}))
\cong
\bigoplus_{i=1}^r H^0(X, \mathcal{O}_X(D)\otimes \CO_X/m_{x_i}^{k_i}).
\]
The set $\{\chi^{u}\,|\,u\in P_D\cap M\}$ is a basis of 
$H^0(X, \mathcal{O}_X(D))$.
For each $1 \leq i \leq r$, let $R_i=\mathbbm{k}[\cone(P_D,u_{i}) \cap M]$ and let $m_i$ be its irrelevant maximal ideal.
For each $1 \leq i \leq r$, by Remark~\ref{remark.maximal.ideal.powers} the space $H^0(X, \mathcal{O}_X(D)\otimes \CO_X/m_{x_i}^{k_i})$ is spanned over $\mathbbm{k}$ by the images of the sections $\chi^u$ such that $\chi^{u-u_i} \in R_i \smallsetminus m_{i}^{k_i}$.
For each $1 \leq i \leq r$, by Proposition~\ref{proposition.onept}, $u \in P_D$ for all such sections.
We claim now that if we fix distinct indices $i$ and $j$, all sections $\chi^u \in H^0(X, \mathcal{O}_X(D))$ such that $\chi^{u-u_i} \in R_i \smallsetminus m_{i}^{k_i}$, map to zero under the composition
\[
H^0(X, \mathcal{O}_X(D))
\rightarrow
H^0(X, \mathcal{O}_X(D)\otimes \CO_X/(m_{x_1}^{k_1}\otimes m_{x_2}^{k_2}\otimes\cdots\otimes m_{x_r}^{k_r}))
\rightarrow 
H^0(X, \mathcal{O}_X(D)\otimes \CO_X/m_{x_j}^{k_j}). 
\]
The surjectivity of the evaluation map in the statement clearly follows from this claim.
To establish the claim, it is enough to fix distinct indices $i$ and $j$, and some $u$ in $P_D$ such that $\chi^{u-u_i} \in R_i \smallsetminus m_{i}^{k_i}$ and show that $\chi^{u-u_j} \in m_j^{k_j} \subseteq R_j$.
Let us choose distinct vertices $v_1,\ldots,v_n$, each connected to $u_i$ by an edge of $P_D$ such that $Q=\cone(v_1-u_i,\ldots,v_n-u_i)$ is simplicial and contains $u-u_i$.
Let $R_Q=\mathbbm{k}[Q \cap M]$ and let $m_Q$ be its irrelevant maximal ideal.
Since $\chi^{u-u_i} \in R_i \smallsetminus m_{i}^{k_i}$, then by Lemma~\ref{lemma.generators}, 
\begin{equation}  \label{bound.weight}
W^Q(u-u_i) \leq W_{\max}^{\cone(P_D,u_{i})}(u-u_i) \leq k_i-1+ \Gamma_{\sigma_i^{\vee}}
\leq k + \Gamma_{\sigma_i^{\vee}} \leq L_{\sigma_i}.
\end{equation}
Since each of the vectors $v_1-u_i,\ldots,v_n-u_i$ has lattice length at least $L_{\sigma_i}$ and $W^Q(u-u_i) \leq L_{\sigma_i}$, 
then $u \in \conv(u_i,v_1,\ldots,v_n)$.
We now consider two cases depending on whether $u_j$ is one of the vertices $\{v_1,\ldots,v_n \}$.

Assume that $u_j \notin \{v_1,\ldots,v_n \}$.
In this case, by Lemma~\ref{lemma.non.adjacent} there exist $\{w_1,\ldots,w_{m} \}$ distinct vertices of $P_D$ each connected to $u_j$ by an edge, such that 
$u \in \cone_{u_j}(w_1,\ldots,w_{m})$, 
$\cone_{u_j}(w_1,\ldots,w_{m})$ is simplicial, 
and $u \notin \conv(u_j,w_1,\ldots,w_{m}) \smallsetminus \conv(w_1,\ldots,w_{m})$.
We set $Q'=\cone(w_1-u_j,\ldots, \linebreak[0] w_{m}-u_j)$. 
By definition of $L_{\sigma_j}$, the vectors $w_1-u_j,\ldots,w_{m}-u_j$ have lattice length at least $L_{\sigma_j}$.
It follows that 
\[
W_{\max}^{\cone(P_D,u_{j})}(u-u_j)
\geq W^{Q'}(u-u_j) \geq L_{\sigma_j} \geq k + \Gamma_{\sigma_j^{\vee}} > k - 1 + \Gamma_{\sigma_j^{\vee}}.
\] 
Therefore, $\chi^{u-u_j} \in m_j^{k_j} \subseteq R_j$ by Lemma~\ref{lemma.generators}, as desired.

Assume now that $u_j \in \{v_1,\ldots,v_n \}$.
Let $w$ be the primitive vector in the direction from $u_i$ to $u_j$. Notice that $\chi^{-w} \in m_j \subseteq R_j$.
To finish the proof, it is enough to show that $u+k_jw$ is in $P_D$.  %
Indeed, this implies that $\chi^{(u-u_j)+k_jw} \in R_j$, and hence 
$\chi^{u-u_j}=\chi^{(u-u_j)+k_jw}(\chi^{-w})^{k_j} \subseteq m_j^{k_j} \subseteq R_j$, as desired to complete the proof.
Let $l$ be the minimum of the lattice lengths of the vectors $v_1-u_i,\ldots,v_n-u_i$.
We saw in (\ref{bound.weight}) that $k_i-1+ \Gamma_{\sigma_i^{\vee}} \geq W^Q(u-u_i)$ and notice that $l \geq L_{\sigma_i} \geq k + \Gamma_{\sigma_i^{\vee}}$,
then   
\[
l-W^Q(u-u_i) 
\geq (k + \Gamma_{\sigma_i^{\vee}}) - (k_i-1+ \Gamma_{\sigma_i^{\vee}}) 
= k+1 - k_i\geq k_j.
\]
Let $v_1',\ldots,v_n'$ be the primitive integer vectors in the directions $v_1-u_i,\ldots,v_n-u_i$ and let $P'=\conv(0,v_1',\ldots,v_n')$.
Notice that $u-u_i \in W^Q(u-u_i)P'$   
and that $w \in \{v_1',\ldots,v_n' \} \subseteq P'$. 
Therefore, 
\begin{align*}
    (u-u_i)+k_jw & \in W^Q(u-u_i)P'+k_jP'  = (W^Q(u-u_i)+k_j)P' \subseteq  lP'  \\   
    & = \conv(0,lv_1',\ldots,lv_n') \subseteq \conv(0,v_1-u_i,\ldots,v_n-u_i).
\end{align*}
Then, $u+k_jw \in \conv(u_i,v_1,\ldots,v_n) \subseteq P_D$, and this completes the proof.
\end{proof}

\subsection{Putting it all together: Proofs of the main results on jet ampleness in  Section~\ref{subsection.theorems}}

\begin{proof}[Proof of Theorem~\ref{theorem.jet.ampleness.technical}]
By Proposition~\ref{proposition.invariant.enough}, we are reduced to show that $D$ generates $k$-jets on any $Z \subseteq X$ supported at $T$-invariant points.
Under the given assumptions, the generation of $k$-jets supported at one $T$-invariant point is proved in Proposition~\ref{proposition.onept} and  the generation of $k$-jets supported at two or more $T$-invariant points is proved in Proposition~\ref{proposition.points}. The desired conclusion follows.
\end{proof}

\begin{proof}[Proof of Theorem~\ref{theorem.jet.ampleness.gamma}]
The equivalence of conditions (\ref{thm.condition.intersection}), (\ref{thm.condition.length}), (\ref{thm.condition.concavity}) and { (\ref{thm.condition.seshadri})} is shown in Proposition~\ref{proposition.triple.equivalence}. 
For any $\sigma \in \Delta$ we have that $\Gamma_X \geq \Gamma_{\sigma^\vee}$ and 
that $L_\sigma \geq k + \Gamma_X$ by (\ref{thm.condition.length}). Hence, $L_{\sigma} \geq k + \Gamma_{\sigma^{\vee}}$, and then $D$ is $k$-jet ample by Theorem~\ref{theorem.jet.ampleness.technical}.
\end{proof}

\begin{proof}[Proof of Corollary~\ref{corollary.jet.ampleness.k+n-2}]
The equivalence of conditions (\ref{corollary.condition.intersection}), (\ref{corollary.condition.length}),  (\ref{corollary.condition.concavity}) and  {   (\ref{corollary.condition.seshadri})} is shown in Proposition~\ref{proposition.triple.equivalence}. 
By Lemma~\ref{lemma.gamma.cone}, in this case $0 \leq \Gamma_X \leq n-2$, and then the conclusion follows from Theorem~\ref{theorem.jet.ampleness.gamma}. The bound is sharp by Example~\ref{example.bounds.sharp} below. 
\end{proof}

\begin{proof}[Proof of Theorem~\ref{theorem.jet.ampleness.multiple}]
The $k$-jet ampleness of $mD$ follows from Example~\ref{example.products}.  
The bound is sharp by Example~\ref{example.bounds.sharp} below. 
\end{proof}



\section{Examples}    \label{section.examples}

In this section, we show that the bounds in Corollary~\ref{corollary.jet.ampleness.k+n-2} and Theorem~\ref{theorem.jet.ampleness.multiple} are sharp and present some sufficient conditions for $k$-jet ampleness on weighted projective spaces. 

\subsection{Sharpness of the bounds in Corollary~\ref{corollary.jet.ampleness.k+n-2} and Theorem~\ref{theorem.jet.ampleness.multiple}}


For $n\geq 2$, we will construct a projective toric variety $X(\Delta)$ with an ample $T$-invariant Cartier divisor $D$ such that the divisor $(k+n-3)D$ does not separate $k$-jets at one invariant point, for each $k \geq 1$. The construction is inspired by an example in \cite{EW91}. 

\begin{example}    \label{example.bounds.sharp}
Let $M=\mathbb{Z}^n$ for some $n \geq 2$ and let $\{e_1,\ldots,e_n\}$ be the its standard basis. Let $r$ be a positive integer such that $r>n-2$ and $a=e_1+\cdots+e_{n-1}+re_n\in M$. Let $P$ be the convex hull of $0,e_1,\ldots, e_{n-1}, a$. Let $X$ be the associated toric variety and $D$ be the Cartier $T$-invariant divisor such that $P_D=P$. 

Let $\sigma$ be the maximal cone corresponding to the vertex $0$ of $P$ and $x$ be the corresponding closed $T$-invariant point. We will show that the Cartier divisor $G:=(k+n-3)D$ does not separate $k$-jets at $x$ for any $k\geq 1$. 
Notice that $L^G_{\sigma}=k+n-3$. Let $Q$ be the dual cone $\sigma^\vee$. We have that $Q$ is generated by $e_1,\ldots,e_{n-1}, a$ in $M\otimes_{\mathbb{Z}}\mathbb{R}$. Then the affine coordinate ring of $U_\sigma$, denoted by $R$, is $\mathbbm{k}[Q\cap \mathbb{Z}^n]$ and its irrelevant ideal $m$ is the maximal ideal of $x$.

Let $a_i=e_1+\cdots+e_{n-1}+ie_n$ for $1\le i\le r$. In particular, $a_r=a$. 
Since \[
S_Q \cap M= \{ {\textstyle \sum_{i=1}^{n-1} } \alpha_{i}e_i+ \alpha_n a \, | \, 0 \leq \alpha_i < 1 \textnormal{ for all $1 \leq i \leq n$} \} \cap M =\{ 0, a_1, \ldots, a_{r-1} \},
\]
it follows that the maximal ideal $m$ is generated by the monomials 
$\{\chi^{e_1}, \ldots , \chi^{e_{n-1}}, \chi^{a_1}, \ldots, \chi^{a_r}\}$.

Let $u=(k-1)e_1+a_1$. This implies that $\chi^u\in m^k$. We will show that $$\chi^u\in m^k\smallsetminus m^{k+1},$$ and $u$ is not in the polytope $P_G=P_{(k+n-3)D}$. This implies that the map \[H^0(X,\mathcal{O}_X(G))\rightarrow H^0(X, \mathcal{O}_X(G)\otimes \mathcal{O}_X/m_x^{k+1})\] is not surjective, and hence that $G$ is not $k$-jet ample. 
\begin{itemize}[leftmargin=*]
    \item[-] Let us show that $u$ is not in the polytope $P_{G}$. We first calculate the weight of $u$ in $Q$. Note that
\[a_1=e_1+e_2+\cdots+e_n={\textstyle \sum_{i=1}^{n-1}}(1-\frac{1}{r})e_i+\frac{1}{r}a.
\]
Hence
$u=ke_1+e_2+\cdots+e_n=(k-1)e_1+\sum_{i=1}^{n-1}(1-\frac{1}{r})e_i+\frac{1}{r}a$.
We thus have 
\[W^{Q}_{\max}(u)=(k-1)+(1-\frac{1}{r})(n-1)+\frac{1}{r}=n+k-2-\frac{n-2}{r}>  
n+k-3.
\]
Recall that the polytope $P_G$ is the convex hull of $\{0,(k+n-3)e_1,\ldots, (k+n-3)e_{n-1}, (k+n-3)a\}$. In particular, for every lattice point $w$ in $P_G$, we have $W^{Q}(w)\le k+n-3$. Then, $u$ is not in $P_{G}$. 

\item[-]
Now we show that $\chi^u \notin m^{k+1}$. Assume that $\chi^u\in m^l$ for some positive integer $l$. We write $u=\sum_{i=1}^{l}w_i$ for some nonzero lattice points $w_i$ in $Q$. 
Each lattice point in $Q$ has the form 
\[{\textstyle 
\sum_{i=1}^{n-1}c_ie_i+\sum_{j=1}^{r}d_ja_j   }
\]
for some integers $c_i\ge 0$ and $d_j\ge 0$.   
Since the $n$-th coordinate of $u$ is $1$, there exists exactly one $w_i$ with positive $n$-th coordinate. We may assume that it is $w_1$. Hence $w_1$ has the form 
$\sum_{i=1}^{n-1}c_ie_i+a_1$
and all the other $w_i$ are sums of nonnegative integer multiples of $e_1, e_2, \ldots, e_{n-1}$. We thus have $w_1-a_1\in Q$ and  \[(k-1)e_1=u-a_1=(w_1-a_1)+{\textstyle \sum_{i=2}^{l}w_i}.\]
Hence each $w_i$ is a multiple of $e_1$. This implies that $l-1\le k-1$. Hence $l\le k$ and $\chi^u \notin m^{k+1}$. 
\end{itemize}

Let us consider the notion of higher concavity of piecewise linear functions on fans in Definition~\ref{definition.concavity} and 
its equivalences in Proposition~\ref{proposition.triple.equivalence}.  
The function $\psi_D$ in this example is $1$-concave because the lattice length of any edge of $P_D$ is exactly equal to $1$. 
This implies that  $\psi_{G}$ is $(k+n-3)$-concave. 
Equivalently, $G\cdot C\ge k+n-3$ holds for each complete $T$-invariant curve $C$ in $X$. This shows that Corollary~\ref{corollary.jet.ampleness.k+n-2} is sharp. Since $G$ is equal to the $k+n-3$ multiple of an ample divisor $D$, we deduce that Theorem~\ref{theorem.jet.ampleness.multiple} is sharp. In fact, Theorem~\ref{theorem.jet.ampleness.technical} is sharp asymptotically. Here
\begin{align*}
\Gamma_{\sigma^\vee}
&=\max\limits_{1 \leq i \leq r}
\left\{ W^Q(a_i)-k_{a_i} \right\}
=\max\limits_{1 \leq i \leq r}
\left\{ (n-1)(1-\frac{i}{r})+\frac{i}{r}-1 \right \}     
=n-2-\frac{n-2}{r}.
\end{align*}
when $r$ is large, $k+\Gamma_{\sigma^\vee}=k+n-2-\frac{n-2}{r}$ has the limit value equal to $k+n-2$. This shows that the $k+\Gamma_{\sigma^\vee}$ is indeed the optimal lower bound of $L_\sigma$ in Theorem~\ref{theorem.jet.ampleness.technical}. 
\end{example}

\subsection{Example: Weighted projective spaces}

\begin{theorem}    \label{theorem.example.weighted.projective.spaces}
Let $\mathbb{P}=\mathbb{P}(a_0,\ldots,a_{n})$ be a weighted projective space with $n\ge 2$, with weights reduced as usual to satisfy $\gcd(a_0,\ldots, \widehat{a_j}, \ldots,a_{n})=1$, for each $0 \leq j \leq n$.
Let $l=\lcm(a_0,\ldots,a_{n})$ and $h=\max\{\lcm(a_i,a_j)\ | \ 0 \leq i < j \leq n \}$.
Then, any ample Cartier divisor $D$ on $\mathbb{P}$ is $k$-jet ample for any  
\[
0 \leq k \leq \frac{l}{h} -n +2.
\]
\end{theorem}
\begin{proof}
It is enough to show that the lattice length of any edge of the polytope $P_{D_0}$ corresponding to an ample $T$-divisor $D_0$ generating $\operatorname{Pic}(\mathbb{P})$ is at least $l/h$.
Indeed, any Cartier divisor is linearly equivalent to $rD_0$ for some $r \in \mathbb{Z}^+$, and $P_{rD_0}=rP_{D_0}$, so the lattice lengths of the edges of $P_{rD_0}$ are at least $r\cdot l/h \geq l/h \geq k+n-2$, and the desired conclusion would follow by Corollary~\ref{corollary.jet.ampleness.k+n-2}.
Consider the polytope $Q$ with vertices $(l/a_0,0,\ldots,0),(0,l/a_1,\ldots,0),\ldots,(0,\ldots,0,l/a_n)$ in the hyperplane $\Lambda = \{(x_0,\ldots,x_{n}) \in \mathbb{R}^{n+1} ~|~a_0x_{0}+\cdots+a_{n}x_{n}=l\}$, with lattice $\Lambda \cap \mathbb{Z}^{n+1}$. 
It is a well-known example in toric geometry that the toric variety associated to $Q$ is $\mathbb{P}$ and that the Cartier $T$-divisor associated  to $Q$ generates $\operatorname{Pic}\mathbb{P}$  
(see Corollary 1.25 in the extended version of \cite{RT12}). 
%
%
Now, for each $0 \leq i < j \leq n$, it is immediate to see that the lattice length of the segment between the $i$-th and $j$-th vertices is precisely $l/\lcm(a_i,a_j)$, and the conclusion follows.
\end{proof}

\begin{example}
Let $\mathbb{P}=\mathbb{P}(a,b,c)$ be a weighted projective plane with weights reduced such that $a,b,c$ are pairwise coprime positive integers.
Let $D$ be an ample Cartier divisor on $\mathbb{P}$.
Then, $D$ is $k$-jet ample for any $0 \leq k\leq \min\{a,b,c\}$ by Theorem~\ref{theorem.example.weighted.projective.spaces}. 
\end{example}



\section{Fujita's jet ampleness conjecture}     \label{section.application}

In this section, we prove that a generalized Fujita's conjecture for k-jet ampleness holds for toric varieties. We first deal with the case that $X$ is the projective space $\mathbb{P}^n$. Let  $H$ be a hyperplane divisor, and $D$ an ample Cartier divisor on $X$. Here $K_X$ is linearly equivalent to $-(n+1)H$. Let $\ell$ be the smallest intersection number of $D$ with curves in $X$. Hence $\ell$ is the degree of $D$. If $\ell\ge n+k+1$, then $\deg(K_X+D)=\ell-n-1\ge k$. Hence $K_X+D$ is $k$-jet ample.  

Now we show that if $X$ is not the projective space, a generalization of Fujita's conjecture for $k$-jet ampleness holds true for smaller intersection numbers, as stated in Theorem~\ref{thm.Fujita}. 
This generalization of Fujita's freeness conjecture was proved by Fujino, see \cite{Fuj03}. The Fujita's very ampleness conjecture was proved by Payne in \cite{Pay06}. In order to prove the higher jets cases for $k\ge 2$, we will need the following strengthening of a nefness result due to Fujino, which can be deduced from  \cite[Theorem 0.1]{Fuj03}. This version was proved by Payne in \cite[Section~4]{Pay06}. Note that $D+D'$ is not assumed to be Cartier here. 

\begin{theorem}[{\cite[Theorem 0.1]{Fuj03},\cite[Section 4]{Pay06}}]\label{Thm.Fujinoplus}
Let $X$ be a projective $n$-dimensional toric variety not isomorphic to $\mathbb{P}^{n}$. Let $D$ and $D^{\prime}$ be $\mathbb{Q}$-Cartier divisors such that $0 \geq D^{\prime} \geq K_{X}$, and $D \cdot C \geq n$ for all $T$-invariant curves $C$. Then $D+D^{\prime}$ is nef.
\end{theorem}

Our proof follows the same idea of the proof of Fujita's very ampleness for singular toric varieties. We briefly describe the idea here. For any maximal cone $\sigma$, let $u_\sigma$ and $u_{\sigma}'$ be the corresponding points of the divisors $D$ and $D'$ in $M_{\mathbb{Q}}$. Since $-D'\le -K_X$ is a small divisor, the polytope $P_{D+D'}$ is obtained by moving the faces of $P_D$ by a small distance. The new vertices of $P_{D+D'}$ are interior points of $P_D$. 

\begin{definition} \label{definition.minimum.weight.function}
Let $Q$ be a cone in a lattice $M$ and let $w_1,\ldots,w_{m}$ be the primitive lattice generators of its rays. 
We define the minimum weight function $W_{\min}^Q: Q \rightarrow \mathbb{R}$ by 
\[W_{\min}^Q(u) := \min\left\{ \left. \sum_{i=1}^{m}a_{i} \,  \right| \, u= \sum_{i=1}^{m}a_{i}w_i \textnormal{ and } a_i \geq 0 \textnormal{ for all $i$} \right\}.
\]
for any $u \in Q$. When $Q$ is simplicial we have ${\displaystyle W_{\min}^Q=W_{\max}^Q}$.  
\end{definition}

In \cite{Pay06}, Payne proved that for each maximal cone $\sigma$, there is a lower bound for the lattice length of edges of $P_{D+D'}$ starting from the vertex $u_{\sigma}+u_{\sigma}'$ in terms of the minimal weight of $u_{\sigma}'$. 

\begin{proposition}[{\cite[Proposition]{Pay06}}]\label{prop.ld.m}
Let $X$ be a complete toric variety and $\sigma$ a maximal cone in the fan defining $X$ and $Q$ the dual cone $\sigma^\vee$. Let $D$ and $D'$ be $T$-invariant $\mathbb{Q}$-Cartier divisors such that $D$ is nef and $0 \geq D^{\prime} \geq K_{X}$. Let $t_{\sigma}=\min \{D \cdot V(\sigma \cap \tau)\}$ and $m_{\sigma}=\min \left\{\left(D+D^{\prime}\right) \cdot V(\sigma \cap \tau)\right\}$, where $\tau$ varies over all maximal
cones adjacent to $\sigma$. Suppose $t_\sigma \geq W^{Q}_{\min}\left(u_{\sigma}^{\prime}\right)$. Then
$$
m_{\sigma} \geq t_{\sigma}-W^Q_{\min}\left(u_{\sigma}^{\prime}\right)-1\ge t_{\sigma}-W^Q_{\max}\left(u_{\sigma}^{\prime}\right)-1.
$$
\end{proposition}
Later we will use the following upper bound of the maximal weight of $u'_{\sigma}$.  
\begin{lemma}[{\cite[Lemma 3]{Pay06}}] \label{lemma.weightof uprime}
Let $D^{\prime}$ be a $\mathbb{Q}$-Cartier $T$-divisor, with $0\geq D' \geq K_{X}$. Then $W^Q_{\max}\left(u_{\sigma}^{\prime}\right) \leq W^Q_{\max}(u)$ for any
lattice point $u$ in the interior of $\sigma^{\vee}.$
\end{lemma}

Now we prove Theorem~\ref{thm.Fujita} on Fujita's jet ampleness. 
\begin{proof}[Proof of Theorem~\ref{thm.Fujita}]
The case $k=0$ is proved by Fujino, see \cite[Corollary 0.2]{Fuj03}. We will assume that $k\ge 1$. 
Notice that ${\frac{n}{n+k}}D \cdot C \geq n$, for all $T$-invariant curves $C$. 
By Theorem~\ref{Thm.Fujinoplus}, we have $\frac{n}{n+k}D+D'$ is nef. For any curve $C$ in $X$, we have 
\begin{equation}\label{ineq-k}
\left(D+D^{\prime}\right) \cdot C=\frac{k}{n+k}(D \cdot C)+\left(\frac{n}{n+k} D+D^{\prime}\right) \cdot C \geq k.    
\end{equation}

Hence $D+D'$ is ample. In particular, $P_{D+D'}$ is a lattice polytope with vertices $u_{\sigma}+u'_\sigma$ for all maximal cones $\sigma$ in the fan of $X$. Let $L_\sigma$ be the minimum of the lattice lengths of the edges of $P_{D+D'}$ having $u_{\sigma}+u'_{\sigma}$ as a vertex. 
For any maximal cone $\sigma'$ sharing an $(n-1)$-dimensional cone $\tau$ with $\sigma$, we have that the lattice length of the edge connecting $u_{\sigma}+u'_{\sigma}$ and $u_{\sigma'}+u'_{\sigma'}$ is the intersection number $(D+D')\cdot V(\tau)$ (see for example the proof of Proposition~\ref{proposition.triple.equivalence}). 
Hence $L_\sigma=m_\sigma$. 

We will show that $m_\sigma\ge k+\Gamma_{\sigma^\vee}$ for every maximal cone $\sigma$. Then the $k$-jet ampleness of $D+D'$ follows from the main theorem. Let $\mu$ be the minimal intersection number of $D+D'$ with $T$-invariant curves in $X$. Since $\mu=\min\{m_\sigma\}$, where $\sigma$ runs over maximal cones of $X$, the inequality \eqref{ineq-k} implies that  $\mu\ge k$, which implies that $m_\sigma\ge k$. 

For simplicity, we will denote $\sigma^\vee$ by $Q$. Let $w_1, w_2, \ldots, w_{m}$ be the primitive lattice generators of the rays of $Q$ and $
S=\{\sum_{i=1}^{m} a_{i} w_{i} \,|\, 0 \leq a_{i}<1 \text { for all } i\}$. Let $p$ be a lattice point in $S$ which computes $\Gamma_Q$. Hence $\Gamma_{\sigma^\vee}=\Gamma_{Q}=W^{Q}_{\max}(p)-k_p$. Now we show that $m_\sigma\ge k+\Gamma_{Q}$. 

\begin{itemize}[leftmargin=*]
    \item[-] If $\sigma$ is smooth, then $\Gamma_{Q}=0$. Hence $m_\sigma\ge \mu\ge k=k+\Gamma_{Q}$.

\item[-] If $\sigma$ is not smooth, i.e., $S\neq \{0\}$.  

\noindent If $p=0$, then $\Gamma_Q=0$. Hence $m_\sigma\ge \mu\ge k=k+\Gamma_Q$.

\noindent If $p$ is not equal to $0$, then $k_p\ge 1$. 
Furthermore, we can choose a top-dimensional simplicial subcone $Q'$ of $Q$ containing $p$ such that $W_{\max}^Q$ is linear on $Q'$. Let $\{w_{i_1}, \ldots, w_{i_n}\}$ be the subset of $\{w_1, w_2, \ldots, w_{m}\}$  containing the primitive lattice generators of the rays of $Q'$. Then $p=a_{1} w_{i_{1}}+\cdots+a_{n} w_{i_n}$ for some $0\le a_i<1$. One can check that $p'=(1-a_{1}) w_{i_{1}}+\cdots+(1-a_{n}) w_{i_n}$ is a lattice point in the interior of $Q'$. Hence it is in the interior of $Q$ and $W^Q_{\max}(p')=W_{\max}^{Q'}(p')=\sum_{i=1}^n (1-a_i)=n-W^Q_{\max}(p)$. Lemma~\ref{lemma.weightof uprime} implies that 
$W^Q_{\max}\left(u_{\sigma}^{\prime}\right) \leq W^Q_{\max}(p')=n-W^Q_{\max}(p)$.
This implies that $W^Q_{\max}(u_\sigma')\le n<n+k\le t_\sigma$.
Proposition~\ref{prop.ld.m} and Lemma~\ref{lemma.weightof uprime} imply that 
\[
m_\sigma \geq n+k-W^Q_{\max}\left(u_{\sigma}^{\prime}\right)-1 
\geq k-1+W^Q_{\max}(p)
= k-1+\Gamma_Q+k_p\ge k+\Gamma_Q.
\qedhere
\] 
\end{itemize}
\end{proof}

\begin{corollary}
Let $D$ be an ample Cartier divisor on a projective $n$-dimensional Gorenstein toric variety $X$ not isomorphic to $\mathbb{P}^n$. Then $K_X+(n+k)D$ is $k$-jet ample for all $k\in \mathbb{Z}_{\ge 0}$.   
\end{corollary}


\section{Higher concavity of piecewise linear functions}   \label{section.concavity}

In this section, we introduce a notion of higher concavity for piecewise linear functions on fans in Definition~\ref{definition.concavity} and study its basic properties in Section~\ref{preliminaries.concavity}. In particular, we relate this higher concavity to properties of divisors on toric varieties in Proposition~\ref{proposition.triple.equivalence}.

\subsection{Review: Multiplicities of cones}   \label{preliminaries.multiplicities}

Recall that if $\sigma$ is a cone and $v_1,\ldots,v_k$ are the first lattice points along its rays,
the multiplicity of $\sigma$, denoted by $\operatorname{mult}(\sigma)$ or $\operatorname{mult}(v_1,\ldots,v_k)$, is defined to be the index of  the lattice generated by the $v_i$ in the lattice associated to $\sigma$:
\[
\operatorname{mult}(\sigma)=\operatorname{mult}(v_1,\ldots,v_k):=[ N_{\sigma}:\mathbb{Z}v_1+\cdots+\mathbb{Z}v_k].
\]

{ The primitive generators of the rays of a smooth cone are part of a basis of the lattice, hence the mul\-ti\-pli\-ci\-ty of any smooth cone is one.} Using multiplicities we can write the following relation between the primitive generators of the rays of a top dimensional simplicial cone and the primitive generators of the rays of its dual.

\begin{lemma}   \label{lemma.multiplicities}
Let $v_1,  \allowbreak \ldots, \allowbreak v_n \allowbreak \in \allowbreak \mathbb{Z}^n$ be primitive vectors that are linearly independent and 
let $w_1, \allowbreak \ldots, \allowbreak w_n \allowbreak \in \allowbreak \mathbb{Z}^n$ be the primitive generators of the rays of  $\operatorname{Cone}(v_1,\ldots,v_n)^{\vee}$, labeled such that $\langle w_i,v_j \rangle=0$ for $i \neq j$. Then, for each $1\leq i \leq n$,
\[
\langle v_i, w_i\rangle=\frac{\operatorname{mult}(v_1,\ldots,v_n)}{\operatorname{mult}(v_1,\ldots,\widehat{v_i},\ldots,v_n)}.
\]
\end{lemma}

\begin{proof}
By symmetry, it is enough to prove the formula in the case $i=1$.
Let $v_2',\ldots,v_n' \in \mathbb{Z}^n$ be a  $\mathbb{Z}$-basis of the lattice $\langle v_2,\ldots,v_n \rangle_{\mathbb{R}}\cap \mathbb{Z}^n$.
Let $v_1'\in \mathbb{Z}^n$ be such that $v_1',\ldots,v_n'$ is a $\mathbb{Z}$-basis of $\mathbb{Z}^n$.
Notice that $\operatorname{mult}(v_1',v_2,\ldots,v_n)=\operatorname{mult}(v_2,\ldots,v_n)$.
Since $w_1$ is primitive, there exists a vector $u \in \mathbb{Z}^n$ such that $\langle w_1,u\rangle=1$.
Since $u= a_1 v_1'+\cdots+ a_n v_n' $ for some $a_i \in \mathbb{Z}$, we have $1 = \langle w_1,u\rangle= a_1 \langle w_1,v_1'\rangle $, and then $|\langle w_1,v_1'\rangle|=1$.
We know that $v_1=c_1v_1'+c_2v_2+\cdots+c_nv_n$ for some $c_i\in \mathbb{Q}$. Hence $\langle w_1,v_1\rangle =c_1 \langle w_1, v_1'\rangle$.    
Let $m$ be a positive integer, divisible enough such that $mc_i\in \mathbb{Z}$ for all $i$. Then we have
\begin{align*}
\operatorname{mult}(v_1,\ldots,v_n)
&=
\frac{1}{m}\operatorname{mult}(m v_1,v_2,\ldots,v_n)
=
\frac{1}{m} \operatorname{mult}(mc_1v_1'+mc_2v_2+\cdots+mc_nv_n,v_2,\ldots,v_n)  \\
&=\frac{1}{m}   
\operatorname{mult}(mc_1v_1',v_2,\ldots,v_n)
=
\frac{|mc_1|}{m} \operatorname{mult}(v_1',v_2,\ldots,v_n)      \\
&= 
|c_1| \operatorname{mult}(v_1',v_2,\ldots,v_n).
\end{align*}
Therefore,
\[
\langle w_1,v_1\rangle 
=\langle w_1,v_1'\rangle c_1
=|\langle w_1,v_1'\rangle| |c_1| 
=|\langle w_1,v_1'\rangle| \frac{\operatorname{mult}(v_1,\ldots,v_n)}{\operatorname{mult}(v_1',v_2,\ldots,v_n)}
=
\frac{\operatorname{mult}(v_1,\ldots,v_n)}{\operatorname{mult}(v_2,\ldots,v_n)}.      
 \qedhere
\]  
\end{proof}

Using multiplicities we can write a simple expression for the class of a given element of a lattice on the quotient of the lattice by the sublattice associated to a codimension-one cone.

\begin{lemma}    \label{definition.equivalence}
Let $\tau$ be a codimension-one cone in $N_\mathbb{R}$ and let $v_0 \in N \smallsetminus N_{\tau}$. 
If $s_0 \cdot e \in \mathbb{Z} \cdot e \cong N/N_{\tau}$ is the class of the image of $v_0$, where the generator $e$ is chosen such that $s_0 \in \mathbb{Z}^+$, then \[
s_0=\frac{\operatorname{mult}(\tau,v_0)}{\operatorname{mult}(\tau)}.
\]
\end{lemma}
\begin{proof}
Let $v_1,\dots,v_r$ be the primitive generators of the rays of $\tau$ 
and let $v_0' \in N$ be a lifting of $e$.
We know that $v_0=s_0v_0'+c_1v_1+c_2v_2+\cdots+c_rv_r$ for some not necessarily unique  $c_i\in \mathbb{Q}$.
Choose $m \in \mathbb{Z}^+$ sufficiently divisible such that $mc_i \in \mathbb{Z}$ for all $i$.
Then we have
\begin{align*}
\operatorname{mult}(\tau,v_0) &=\operatorname{mult}(v_0,v_1,\ldots,v_r)
=
\frac{1}{m}\operatorname{mult}(mv_0, v_1,v_2,\ldots,v_r)
\\
&=\frac{1}{m}   
\operatorname{mult}(ms_0v_0',v_1,\ldots,v_r)
=
\frac{ms_0}{m} \operatorname{mult}(v_0',v_1,\ldots,v_r)      \\
&= 
s_0 \cdot \operatorname{mult}(v_1,\ldots,v_r)
=
s_0 \cdot \operatorname{mult}(\tau).  \qedhere
\end{align*}
\end{proof}


We can also use multiplicities to evaluate some intersection numbers on toric varieties. Related statements can be found in \cite[Section 1.6]{laterveer1996} and \cite[Lemma 2]{Pay06}.  

\begin{lemma}\label{lemma.intersection}
Let $D$ be a Cartier $T$-divisor  on a complete toric variety $X$. For each maximal cone $\sigma\in \Delta$, let $u_\sigma\in M_\mathbb{Z}$ be the lattice point such that $\psi_D(v)=\langle u_\sigma, v\rangle$, for any $v\in \sigma$. Let $\sigma_1$ and $\sigma_2$ be any maximal cones sharing a facet $\tau$. Then for any primitive generator $v_0$ of a ray of $\sigma_2$ not in $\sigma_1$, we have the intersection number
\[D\cdot V(\tau)=\langle u_{\sigma_1}-u_{\sigma_2},v_0\rangle \cdot \frac{\operatorname{mult}(\tau)}{\operatorname{mult}(\tau,v_0)}.\]
\end{lemma}

\begin{proof}
Let $\tilde{D}$ be the divisor $D+\Div(\chi^{u_{\sigma_1}})$, which is linearly equivalent to $D$.  One can check that $\psi_{\tilde D}=\psi_D-u_{\sigma_1}$. Since $\tilde D$ is a trivial divisor on $U_{\sigma_1}$. Then $\tilde D$ intersects $V(\tau)$ only at the closed point $x_{\sigma_2}$.   
By \cite[Section 5.1]{Fulton}, we compute the intersection number as   
\[D\cdot V(\tau)=\tilde D \cdot V(\tau)=\frac{a_0}{s_0},\]
where $a_0=-\psi_{\tilde D}(v_0)=\langle u_{\sigma_1}-u_{\sigma_2}, v_0\rangle$
and $s_0$ is as in Lemma~\ref{definition.equivalence}.
By Lemma~\ref{definition.equivalence},
$s_0=\frac{\mult(\tau,v_0)}{\mult(\tau)}$  
and the conclusion follows.
\end{proof}


\subsection{Higher concavity of piecewise linear functions}    \label{preliminaries.concavity}
We will now introduce the notion of $k$-concavity for piecewise linear functions on fans.
This definition of $k$-concavity is a weighted variation of the notion of $k$-convexity introduced by Di Rocco in \cite{DiRoccoJets99} to study jet ampleness on smooth toric varieties.

\begin{definition}[$k$-concavity]    \label{definition.concavity}
Let $\Delta$ be a fan on $N_{\mathbb{R}}$ with convex full-dimensional support 
and 
let $\psi: \operatorname{Supp}(\Delta) \rightarrow \mathbb{R}$ be a function that is linear on each cone. 
For each maximal cone $\sigma \in \Delta$, let $u_{\sigma} \in M_{\mathbb{R}}$ be such that 
$\psi(v)=\langle u_{\sigma},v \rangle$, for each $v \in \sigma$.
Given $k \in { \mathbb{Q}_{\geq 0}}$, we will say that the function $\psi$ is 
$k$-concave if given any two maximal cones $\sigma_1$ and $\sigma_2$ sharing a facet $\tau$, and each primitive generator $v_0$ of a ray of $\sigma_2$ not in $\sigma_1$, we have
\begin{equation}   \label{definition.inequality}
\langle u_{\sigma_1}, v_0 \rangle   
\  \geq   \
\psi(v_0) + k \cdot s_0,
\end{equation}
where if $e$ denotes the generator of the one-dimensional lattice $N/N_{\tau}$ such that the image of $v_0$ in $N/N_{\tau}$ is a positive multiple of $e$, then $s_0$ is the integer such that $v_0$ maps to $s_0\cdot e$ in $N/N_{\tau}$.
\end{definition}

\begin{remark}\label{remark.definition}
\begin{enumerate}[wide, labelwidth=!, labelindent=0pt]
    \item[\textnormal{($a$)}] The inequality (\ref{definition.inequality}) in Definition~\ref{definition.concavity} can alternatively be presented as
\[
\langle u_{\sigma_1}, v_0 \rangle   
\  \geq   \
\psi(v_0) + k \cdot \frac{\operatorname{mult}(\tau,v_0)}{\operatorname{mult}(\tau)}.
\]
since $s_0 = {\operatorname{mult}(\tau,v_0)}/{\operatorname{mult}(\tau)}$ by Lemma~\ref{definition.equivalence}.

\item[\textnormal{($b$)}] In the setting of Definition~\ref{definition.concavity}, for fixed $\sigma_1$, $\sigma_2$ and $\tau$, 
the inequalities in (\ref{definition.inequality}) hold for each $v_0$ if and only if they hold for a single $v_0$, since the quantity \[
\frac{1}{s_0}\left[{\langle u_{\sigma_1}, v_0 \rangle} - \psi(v_0) \right] 
=
\frac{1}{s_0}{\langle u_{\sigma_1}-u_{\sigma_2}, v_0 \rangle}
=
{\langle u_{\sigma_1}-u_{\sigma_2}, e \rangle}
\]
is constant as $v_0$ ranges over the primitive generators of the rays of $\sigma_2$ not in $\sigma_1$. 
\end{enumerate}
\end{remark}

\begin{remark} The following properties of $k$-concavity follow at once from the definition.
\begin{itemize}[leftmargin=5mm]
    \item[-] If $\psi$ is $k$-concave for some $k\in {\mathbb{Q}_{\geq 0}}$, then $\psi$ is $k'$-concave for all {rational numbers} $0 \leq k' \leq k$.
    \item[-] If $\psi_1$ and $\psi_2$ are respectively $k_1$-concave and $k_2$-concave for some $k_1,k_2 \in {\mathbb{Q}_{\geq 0}}$, then $\psi_1+\psi_2$ is $(k_1+k_2)$-concave.
    \item[-] If $\psi$ is $k$-concave for some $k\in {\mathbb{Q}_{\geq 0}}$ and $m \in {\mathbb{Q}_{\geq 0}}$, then $m\psi$ is $(mk)$-concave.
\end{itemize}
\end{remark}


The following lemma shows that for the piecewise linear function $\psi_D$ associated to an ample $T$-divisor $D$ on a projective toric variety,  
our notion of $k$-concavity in Definition~\ref{definition.concavity} agrees with that of $k$-convexity {introduced in \cite[Definition 3.2]{DiRoccoJets99} when the toric variety is smooth and $k \in \mathbb{Z}_{\geq 0}$}.


\begin{lemma}
Let $\Delta$ be a smooth projective fan on $N_{\mathbb{R}}$ and 
let $\psi_D: \operatorname{Supp}(\Delta)\rightarrow \mathbb{R}$ be the piecewise linear function associated to an ample $T$-divisor $D$ on $X(\Delta)$.
For each maximal cone $\sigma \in \Delta$, let $u_{\sigma} \in M_{\mathbb{R}}$ be such that 
$\psi_D(v)=\langle u_{\sigma},v \rangle$, for each $v \in \sigma$.
Then, the function $\psi_D$ is $k$-concave if and only if for any maximal cone $\sigma \in \Delta$ and every $v_0 \in N$ such that $v_0 \notin \sigma$, one has
\begin{equation}  \label{inequalities.lemma}
\langle u_{\sigma}, v_0 \rangle   
\  \geq   \
\psi_D(v_0) + k.
 \end{equation}
\end{lemma}
\begin{proof}
The multiplicities of all cones in a smooth fan are equal to one.
Then clearly, $\psi_D$ is $k$-concave if and only if for any
two maximal cones $\sigma_1$ and $\sigma_2$ sharing a facet $\tau$, 
and every $v_0 \in N \cap \sigma_2$ such that $v_0 \notin \sigma_1$, one has $\langle u_{\sigma_1}, v_0 \rangle \geq 
\psi_D(v_0) + k$.
Then, if $\psi_D$ satisfies the inequalities in (\ref{inequalities.lemma}) it is $k$-concave.
Reciprocally, suppose that the function $\psi_D$ is $k$-concave.
Let $\sigma_1$ be a maximal cone in $\Delta$ and let $v_0$ in $N$ such that $v_0 \notin \sigma_1$. 
Let $\sigma_2$ be a maximal cone in $\Delta$ such that $v_0 \in \sigma_2$.
Since $D$ is ample, as $\sigma$ ranges over the maximal cones in $\Delta$, the points $u_\sigma$ are all distinct and they are the vertices of the polytope $P_D$.
Let $\sigma_1',\ldots,\sigma_n'$ be the maximal cones in $\Delta$ that share a facet with $\sigma_2$.
Since $\psi_D$ is $k$-concave, we have that $\langle u_{\sigma_i'}, v_0 \rangle \geq 
\psi_D(v_0) + k = \langle u_{\sigma_2}, v_0 \rangle + k$, for each $i$.
The points $u_{\sigma_1'},\ldots,u_{\sigma_n'}$ are precisely the vertices of the polytope $P_D$ that share an edge with $u_{\sigma_2}$, and they are all contained in the half-space $\{ u \in M_\mathbb{R} \ | \ \langle u  , v_0  \rangle  \geq \langle u_{\sigma_2}  , v_0  \rangle +k   \}$. 
Then all vertices of $P_D$ except possibly $u_{\sigma_2}$ are in this half-space.
Therefore, $\langle u_{\sigma_1}, v_0 \rangle \geq  \langle u_{\sigma_2}, v_0 \rangle + k  =  \psi_D(v_0) + k$.
\end{proof}


The following lemma shows that $k$-concavity is a strengthening of the notion of concavity.

\begin{lemma}      \label{zero.concave.same.as.concave}
Let $\Delta$ be a fan on $N_{\mathbb{R}}$ 
whose support is convex and full-dimensional
and let $\psi: \operatorname{Supp}(\Delta)\rightarrow \mathbb{R}$ be a function that is linear on each cone.
Then, $\psi$ is $0$-concave if and only if $\psi$ is a concave function. In particular, if $\psi$ is $k$-concave for any $k\in {\mathbb{Q}_{\geq 0}}$, then it is concave.
\end{lemma}
\begin{proof}
For each maximal cone $\sigma \in \Delta$, let $u_{\sigma} \in M_{\mathbb{R}}$ be such that 
$\psi(v)=\langle u_{\sigma},v \rangle$, for each $v \in \sigma$.

Assume that $\psi$ is $0$-concave.
We want to show that the function $\Phi: \operatorname{Supp}(\Delta) \times \operatorname{Supp}(\Delta) \times [0,1]\longrightarrow \mathbb{R}$ defined by $\Phi(x,y,t)=\psi(tx+(1-t)y)-t\psi(x)-(1-t)\psi(y)$ only takes values greater than or equal to zero.
Let $W$ be the subset of $\operatorname{Supp}(\Delta) \times \operatorname{Supp}(\Delta)$
consisting of the pairs $(x,y)$ such that the line segment joining $x$ and $y$ intersects the interior of each cone that it intersects, but it does not intersect any cone in $\Delta$ of codimension at least two.
By looking at the dimension of its complement, we see that $W$ is dense in $\operatorname{Supp}(\Delta) \times \operatorname{Supp}(\Delta)$. 
Since $\Phi$ is continuous, it is enough to show that $\Phi|_{W \times [0,1]}$ only takes values greater than or equal to zero.
We prove this by induction on the number $r$ of cones whose interior is intersected by the line segment joining the points $x$ and $y$ of a given triple $(x,y,t)$ in $ W \times [0,1]$.
The claim holds for $r=1$ since $\psi$ is linear on each cone.
Let us see the claim also holds for $r=2$. 
For this, it is enough to show the claim for any given $(x,y,t)$ in $W \times [0,1]$, such that the line segment between $x$ and $y$ is contained in two cones $\sigma_1$ and $\sigma_2$ sharing a common facet $\tau$, such that $x \in \sigma_1$ and $y$ in $\sigma_2$.
By symmetry, we may assume that $tx+(1-t)y \in \sigma_2$.
Then we have, 
\begin{align*}
\psi(tx+(1-t)y)
&=
\langle u_{\sigma_2} ,  tx+(1-t)y \rangle 
=
t \langle u_{\sigma_2} , x \rangle + (1-t) \langle u_{\sigma_2} ,  y \rangle   \\
&\geq
t\psi(x)+ (1-t) \langle u_{\sigma_2} ,  y \rangle
=
t\psi(x)+ (1-t) \psi(y).
\end{align*}
For the inductive step, suppose that for some $r\geq 2$ we are given $(x,y,t)$ in $W \times [0,1]$, such that the line segment between $x$ and $y$ intersects the interiors of precisely the maximal cones $\sigma_1,\ldots,\sigma_r$, which are such that each $\sigma_i$ and $\sigma_{i+1}$ share a common facet.
After reparameterization of the linear segment between $x$ and $y$, the conclusion now follows by induction from the simple fact that if $a<b<c<d$ are real numbers and  $f:[a,d]\rightarrow \mathbb{R}$ is a function that is concave on $[a,c]$ and $[b,d]$, then it is concave on the whole $[a,d]$.

Reciprocally, let us assume that $\psi$ is concave.
Let $\sigma_1$ and $\sigma_2$ be two maximal cones sharing a facet 
and let $v_0$ be a primitive generator of a ray of $\sigma_2$ not in $\sigma_1$. 
Let $v_0'$ be a point in the interior of $\sigma_1$ and choose a number $0 < t < 1$ such that $tv_0+(1-t)v_0' \in \sigma_1$.
Then we have,
\begin{align*}
t \langle u_{\sigma_1}  , v_0  \rangle +
(1-t) \langle u_{\sigma_1}  , v_0'  \rangle
&=
\langle u_{\sigma_1}  , tv_0+(1-t)v_0'  \rangle
=
\psi(tv_0+(1-t)v_0')      \\
&\geq  
t \psi(v_0) + (1-t)\psi(v_0')      
=
t \psi(v_0) + (1-t) \langle u_{\sigma_1}  , v_0'  \rangle.
\end{align*}
It follows that $\langle u_{\sigma_1}, v_0 \rangle \geq 
\psi(v_0)$.
Therefore, $\psi$ is $0$-concave.
\end{proof}

Now we relate higher concavity to properties of ample divisors on projective toric varieties. 
{ We refer to Definition~\ref{definition.seshadri} to review the notion of the Seshadri constant.} 

\begin{proposition}    \label{proposition.triple.equivalence}
Let $D$ be an ample Cartier $T$-divisor on a projective toric variety $X=X(\Delta)$. For any nonnegative {rational number} $k$, the following statements are equivalent: 

\begin{enumerate}
    \item \label{condition.intersection} $D \cdot C \geq k$ for each complete $T$-invariant curve $C$ in $X$,
    \item \label{condition.length} Each edge of the polytope $P_D$ has lattice length at least $k$,
    \item \label{condition.concavity} The piecewise linear function $\psi_D$ is $k$-concave, 
    \item \label{condition.seshadri}  { The Seshadri constant $\varepsilon(D)$ is at least $k$.}
\end{enumerate}
\end{proposition}
\begin{proof}
Since (\ref{condition.intersection}), (\ref{condition.length}), (\ref{condition.concavity}) and {(\ref{condition.seshadri})} hold for $k=0$, we assume that {$k >0$}. 
For each maximal cone $\sigma \in \Delta$, let $u_{\sigma} \in M_{\mathbb{R}}$ be such that 
$\psi_D(v)=\langle u_{\sigma},v \rangle$, for each $v \in \sigma$.
We first show that (\ref{condition.intersection}) implies (\ref{condition.concavity}). Assume that $D \cdot C\ge k$ for each complete $T$-invariant curve $C$ in $X$. 
Let $\sigma_1$ and $\sigma_2$ be any two maximal cones sharing a facet $\tau$, $v_0$ be a primitive generator of a ray of $\sigma_2$ not in $\sigma_1$, 
and $s_0$ as before. 
Lemma~\ref{lemma.intersection} implies that 
\[\langle u_{\sigma_1}-u_{\sigma_2},v_0\rangle \cdot \frac{1}{s_0}=D\cdot V(\tau)\ge k.\]
Equivalently, $\langle u_{\sigma_1}-u_{\sigma_2},v_0\rangle \ge k\cdot s_0$.
Note that $\psi_D(v_0)=\langle u_{\sigma_2}, v_0\rangle$. 
This implies that
\[
\langle u_{\sigma_{1}}, v_{0}\rangle
\geq \langle u_{\sigma_{2}}, v_{0}\rangle +k \cdot s_{0}= \psi\left(v_{0}\right)+k \cdot s_{0}.
\]

Now we show that (\ref{condition.concavity}) implies (\ref{condition.intersection}). Let $C$ be a $T$-invariant curve. Since there is a one to one correspondence between $T$-invariant curves and $(n-1)$-dimensional cones. There exists an  $(n-1)$-dimensional cone $\tau$ such that $C=V(\tau)$. Since $X$ is complete, then there are exactly two maximal cones containing $\tau$. We denote them by $\sigma_1$ and $\sigma_2$. The above argument also shows that if $\psi_D$ is $k$-concave, then $D\cdot C=D\cdot V(\tau)\ge k$. 

Now we show that (\ref{condition.intersection}) is equivalent to (\ref{condition.length}).  
Recall that there is a one to one correspondence between the maximal cones in $\Delta$ and the vertices of $P_D$. Any pair of maximal cones $\sigma_1$ and $\sigma_2$, sharing a facet $\tau$, corresponds to an edge of $P_D$, which is equal to the line segment $\overline{u_{\sigma_1}u_{\sigma_2}}$. Let $\gamma$ be the lattice length of this edge. We show that $\gamma=D\cdot V(\tau)$.

Let $v_0 \in N$ be the primitive generator of some ray of $\sigma_2$ not in $\sigma_1$.  
Let $s_0 \cdot e \in \mathbb{Z} \cdot e \cong N/N_{\tau}$ be the class of the image of $v_0$, where the generator $e$ is chosen such that $s_0 \in \mathbb{Z}^+$.
Let $u \in M$ be the primitive vector in the direction $u_{\sigma_1}-u_{\sigma_2}$ and let $v_0' \in N$ be a lifting of $e$.  
Note that $u_{\sigma_1}-u_{\sigma_2} \in \tau^\perp$. 
Since $u$ is primitive, there exists $w \in N$ such that $\langle u, w \rangle=1$. Let $a\cdot e$ be the class of $w$ in $N/N_{\tau}$. 
The images of $w-av_0'$ and $v_0-s_0v_0'$ are zero in $N/N_{\tau}$. Hence $w-av_0', v_0-s_0v_0' \in N_\tau$. 
In particular,  
$1=\langle u, w\rangle =\langle u, av_0'\rangle=a\langle u, v_0'\rangle$, and then $|\langle u, v_0'\rangle| = 1$.  
Notice that $D\cdot V(\tau) > 0$, since $D$ is ample. Now, using Lemmas~\ref{definition.equivalence} and \ref{lemma.intersection}, we have  
$$
\gamma=|\langle \gamma u, v_0'\rangle|=|\langle u_{\sigma_1}-u_{\sigma_2},v_0' \rangle|
=|\langle u_{\sigma_1}-u_{\sigma_2},v_0\rangle| \cdot \frac{1}{s_0}
=|\langle u_{\sigma_1}-u_{\sigma_2},v_0\rangle| \cdot\frac{\mult(\tau)}{\mult(\tau, v_0)}=D\cdot V(\tau).
$$

This proves that (\ref{condition.intersection}) is equivalent to (\ref{condition.length}). 
{ 

Now we show that (\ref{condition.length}) is equivalent to (\ref{condition.seshadri}).
Recall that the Seshadri constant of an ample Cartier $T$-divisor at a $T$-invariant point is the minimum of the lengths of the edges through the corresponding vertex in the associated polytope \cite[Remark~3.14]{Ito12}.
The article \cite{Ito12} is written over the field of complex numbers, but his argument to obtain \cite[Remark~3.14]{Ito12} runs for any algebraically closed field.
Assuming (\ref{condition.length}), then $\min\{ \varepsilon(D;p) \, | \,  p \in X \textnormal{ is $T$-invariant} \} \geq k$ by \cite[Remark~3.14]{Ito12} 
and then 
(\ref{condition.seshadri}) follows by Corollary~\ref{corollary.seshadri.invariant.points}. 
Similarly, if we assume (\ref{condition.seshadri}), then $\min\{ \varepsilon(D;p) \, | \,  p \in X \textnormal{ is $T$-invariant} \} \geq k$, and hence (\ref{condition.length}) follows by  \cite[Remark~3.14]{Ito12}. 
}
\end{proof}


\section{Lower semi-continuity of Seshadri constants on projective toric varieties}   \label{section.Seshadri}

In this section, we establish the lower semi-continuity in the Zariski topology of the Seshadri constant function on projective toric varieties in Theorem~\ref{theorem.lower.semicontinuous}. As a consequence Corollary~\ref{corollary.seshadri.invariant.points} tells us that we can compute the global Seshadri constant from the values at the $T$-invariant points. 
We start by recalling the definition of the Seshadri constant introduced by Demailly.

\begin{definition} \label{definition.seshadri} Let $X$ be a projective variety and $D$ an ample Cartier divisor on $X$. The Seshadri constant $\varepsilon(X,D;p)=\varepsilon(D;p)$ of $D$ at the point $p \in X$ is $\max\{\lambda \in \mathbb{R}_{\geq 0} \, | \, \pi^{*}_{p}D - \lambda E_p \textnormal{ is nef} \} $, where $\pi_p: \operatorname{Bl}_p X \rightarrow X$ is the blowup of $X$ along $p$ and $E_p$ is the exceptional divisor of $\pi_p$. The Seshadri constant $\varepsilon(X,D)=\varepsilon(D)$ of $D$ is $\inf\{ \varepsilon(X, D;p) \, | \,  p \in X\}$. 
\end{definition}


\begin{remark}   \label{remark.fibers}
Let $p$ be a smooth closed point in a variety $X$. 
Let $\pi_1:X\times X \rightarrow X$ be the first projection and let $\pi: \operatorname{Bl}_{\Delta_X} X\times X \rightarrow X \times X$ be the blowup of $X\times X$ along the diagonal $\Delta_X$. 
Then, the scheme theoretic intersection $\pi_1^{-1}(p) \cap \Delta_X$ of the diagonal $\Delta_X$ with the fiber $\pi_1^{-1}(p)$ is the reduced point $(p,p)$ in $X \times X$. 
Also, the fiber $(\pi_{1} \circ \pi)^{-1}(p)$ is isomorphic to the blowup of $\{p \} \times X $ along $\pi_1^{-1}(p) \cap \Delta_X$, hence it is isomorphic to $\operatorname{Bl}_{p} X$.
Moreover, the exceptional divisor of $\pi$ restricts to the exceptional divisor of $\operatorname{Bl}_{p} X$. 
For the proof, one reduces to the case that $X$ is smooth and affine, where these statements easily follow after writing each claim in local coordinates.  
\end{remark}


\begin{lemma}\label{generic-point}
Let $D$ be an ample Cartier $T$-divisor on a projective toric variety $X$. Let $p$ be a smooth closed point of $X$ and $q$ a closed point in the torus $T$ of $X$. Then 
\[\varepsilon(X,D;p)\le \varepsilon(X,D;q).\]
\end{lemma}
\begin{proof}
Let $U$ be the smooth locus of $X$. 
Let $\Delta_U$ be the intersection of the diagonal of $X \times X$ with $U \times X$, considered as a closed subvariety of $U \times X$. 
Let $\pi: \operatorname{Bl}_{\Delta_U} (U \times X) \rightarrow U \times X$ be the blowup of $U \times X$ along $\Delta_U$ and $E$ be the exceptional divisor of $\pi$.
Let $\pi_U$ and $\pi_X$ be the projection morphisms from $U \times X$ onto $U$ and $X$, respectively. 
If we consider the diagonal action of $T$ on $U \times X$, then $\Delta_U$ is $T$-invariant. 
Hence, we get an induced $T$-action on $\operatorname{Bl}_{\Delta_U} (U \times X)$. 
The morphisms $\pi$, $\pi_U$ and $\pi_X$ are $T$-equivariant with respect to these actions. 
%
%
Fix a $T$-invariant Cartier divisor $A_X$ on $X$, sufficiently ample so that the $T$-invariant divisor $A=\pi^*\pi_X^*A_X-E$ on $\operatorname{Bl}_{\Delta_U} (U \times X)$ is ample when restricted to each fiber of the projective surjective morphism $\pi_U \circ \pi$.  

Let $G$ be a $\mathbb{Q}$-Cartier $T$-invariant $\mathbb{Q}$-divisor on $\operatorname{Bl}_{\Delta_U} (U \times X)$ whose support does not contain any of the fibers of $\pi_U \circ \pi$.
For each $m \in \mathbb{Z}^+$, the set $V_m$ consisting of the points $x \in U$ such that $\left.\left(G+\frac{1}{m}A\right)\right|_{(\pi_U \circ \pi)^{-1}(x)}$ is ample on $(\pi_U \circ \pi)^{-1}(x)$ is $T$-invariant and it is open by \cite[Theorem 1.2.17]{PAG}. 
Notice that $G|_{(\pi_U \circ \pi)^{-1}(x)}$ is nef on $(\pi_U \circ \pi)^{-1}(x)$ if and only if $x \in \bigcap_{m \in \mathbb{Z}^+} V_m$. 
We have that $V_m \supseteq V_{m+1}$ for all $m \in \mathbb{Z}^+$ and since we are on the toric variety $U$, the sequence $V_m$ stabilizes to an open $T$-invariant subset $V=\bigcap_{m \in \mathbb{Z}^+} V_m$ of $U$. 
Therefore, the set consisting of the points $x \in U$ such that $G|_{(\pi_U \circ \pi)^{-1}(x)}$ is nef on $(\pi_U \circ \pi)^{-1}(x)$ is open and $T$-invariant.

We can assume that $\varepsilon(X,D;p)>0$ and fix an integer $m_0$ such that $\varepsilon(X,D;p) -  \frac{1}{m_0} > 0$, for example $m_0:= \left\lceil{\frac{1}{\varepsilon(X,D;p)}}\right\rceil$. 
For each integer $m \geq m_0$, the set 
\[
N_m = \left\{ x \in U \, \left| \, \left. \left(\pi^*\pi_X^*D-\left(\varepsilon(X,D;p) -  \frac{1}{m}\right)E\right)\right|_{{(\pi_U \circ \pi)}^{-1}(x)} \textnormal{ is not nef on $(\pi_U \circ \pi)^{-1}(x)$} \right.  \right\}
\]
is $T$-invariant and it is closed in $U$ by the previous paragraph.
Notice that by Remark~\ref{remark.fibers},
$N_m = \{ x \in U \, | \, \varepsilon(X,D;x) < \varepsilon(X,D;p) - \frac{1}{m}  \}$.
We have that $N_m \subseteq N_{m+1}$ for all $m \in \mathbb{Z}^+$ and since we are on the toric variety $U$, the sequence $N_m$ stabilizes to a closed $T$-invariant subset $N$ of $U$, which by construction satisfies $N= \bigcup_{m \geq m_0} N_m =\{ x \in U \, | \, \varepsilon(X,D;x) < \varepsilon(X,D;p) \}$.

Therefore, the nonempty $T$-invariant open set $U \smallsetminus N =\{ x \in U \, | \, \varepsilon(X,D;p) \leq \varepsilon(X,D;x) \}$ intersects the torus, and hence contains the whole torus, as desired. 
\end{proof}


\begin{definition}
Let $P$ be an integral polytope in $M_{\mathbb{R}}$. For each edge $\rho$ in $P$, we denote by $M_\rho$ the lattice $\mathbb{R}(\rho-\rho) \cap M$ in the subspace $\mathbb{R}(\rho-\rho)$ of $M_{\mathbb{R}}$.
We will write $|\cdot|_{M_{\rho}}$ for the lattice length function on $\mathbb{R}(\rho-\rho)$ with respect to $M_{\rho}$.
The lattice length of $\rho$ denoted by $|\rho|_{M_{\rho}}$ is defined to be
the lattice length with respect to $M_{\rho}$ of a translation of $\rho$ contained in $\mathbb{R}(\rho-\rho)$.
For any vertex $v$ of $P$, we define $s(P;v)$ as the minimum of the lattice lengths of the edges through $v$, that is, 
\[
s(P;v):=\min\{|\rho|_{M_{\rho}}~|~ \textnormal{$\rho$ is an edge of $P$ having $v$ as a vertex} \} \in \mathbb{Z}^+.
\]
\end{definition}

Let $D$ be an ample Cartier $T$-divisor on a projective toric variety $X=X(\Delta)$ and $\xi$ a face of the polytope $P_D$. 
We denote by $O_\xi$ the $T$-orbit in $X$ corresponding to $\xi$ and   
we denote by $X_\xi$ the closure of $O_\xi$ in $X$. 
Let $D_\xi$ be an ample Cartier invariant divisor on the toric subvariety $X_{\xi}$ such that $\mathcal{O}_{X_\xi}(D_\xi)=\mathcal{O}_X(D)|_{X_\xi}$. 
Let $\pi_\xi: M_\mathbb{R}\rightarrow M_\mathbb{R}/\mathbb{R}(\xi-\xi)$ be the projection map and set $P'_\xi=\pi_{\xi}(P_D)$ and $v'_\xi=\pi_\xi(\xi)$. Ito gave the following formula to calculate the Seshadri constant of $D$ at every closed point of $X$.


\begin{proposition}[{\cite[Proposition 3.12]{Ito12}}] \label{ito's prop3.12}
Let $\xi$ be a face of an $n$-dimensional integral polytope $P\subset M_{\mathbb{R}}$. Let $X$ be the projective toric variety associated to $P$ and $D$ be the ample Cartier $T$-divisor on $X$ corresponding to $P$. Then, for any $q \in O_\xi$ one has
\[\varepsilon(X,D;q)=\min\{\varepsilon(X_\xi,D_\xi;q),s(P'_\xi;v'_\xi)\}.\]
\end{proposition}

\begin{remark} The article \cite{Ito12} is written over the field of complex numbers, but his argument to obtain \cite[Proposition 3.12]{Ito12} runs for any algebraically closed field (one considers points in the torus orbit instead of very general points, and uses Lemma~\ref{generic-point} instead of \cite[Example 5.1.11]{PAG}).
\end{remark}


\begin{theorem}    \label{theorem.lower.semicontinuous}
Let $D$ be an ample Cartier $T$-divisor on a projective toric variety $X=X(\Delta)$. The Seshadri constant function 
\begin{align*}
\varepsilon(D;-):X &\longrightarrow \mathbb{R}\\
 p &\longmapsto \varepsilon(D;p)   
\end{align*}
is lower semi-continuous in the Zariski topology of $X$.  
\end{theorem}

\begin{proof}
Let $\xi$ be a face of $P_D$ such that $\dim \xi\leq n-1$ and $\tau$ be a face of $P_D$ such that $\xi \subseteq \tau$ and $\dim \tau=\dim \xi+1$. Then for any $p\in O_\xi$ and $q\in O_{\tau}$, we claim that  \[\varepsilon(X,D;p)\le \varepsilon(X,D;q).\] 
Since the Seshadri constant is constant on $T$-orbits, our assertion follows from the above inequality. 

Now we prove the claim. If $\dim\xi=n-1$, then $p$ is a smooth point in $X$.  The desired inequality follows from Lemma \ref{generic-point}. We may assume that $\dim\xi\le n-2$.  
Set $M':=\pi_{\xi}(M)$, $P'_{\xi}:=\pi_{\xi}(P_D)$ and $v'_\xi:=\pi_\xi(\xi)$.
The point $v'_\xi$ is a vertex of the polytope $P'_{\xi}$.  
Let $\pi_{\tau}$ be the projection $M_\mathbb{R}\rightarrow M_\mathbb{R}/\mathbb{R}(\tau-\tau)$ and set $M'':=\pi_{\tau}(M)$, $P''_\tau:=\pi_{\tau}(P_D)$ and $v''_\tau:=\pi_\tau(\tau)$. 
The point $v''_\tau$ is a vertex of the polytope $P''_{\tau}$. The projection map $\pi_{\tau}$ factors through the projection map $\alpha: M_\mathbb{R}/\mathbb{R}(\xi-\xi)\rightarrow M_\mathbb{R}/\mathbb{R}(\tau-\tau)$. 
We have that $M''=\alpha(M')$, $P''_{\tau}=\alpha(P'_{\xi})$ and $v''_{\tau}=\alpha(v'_{\xi})$.

We show that $s(P'_\xi;v'_\xi)\leq s(P''_\tau;v''_\tau)$. Let $\eta''$ be an edge of $P''_{\tau}$ which computes  $s(P''_\tau;v''_\tau)$. Note that  $\alpha^{-1}(\mathbb{R}(\eta''-\eta''))$ is a $2$-dimensional vector space and $\alpha$ maps $\pi_\xi(\tau)$ to the point $v''_\tau$. 
Hence $\alpha^{-1}(\eta'') {\cap P'_\xi}$ is a $2$-dimensional face of $P'_\xi$. Let $\rho'$ be the edge of $\alpha^{-1}(\eta''){\cap P'_\xi}$ passing through $v'_{\xi}$, which is different from $\pi_{\xi}(\tau)$. This is represented in the following figure.
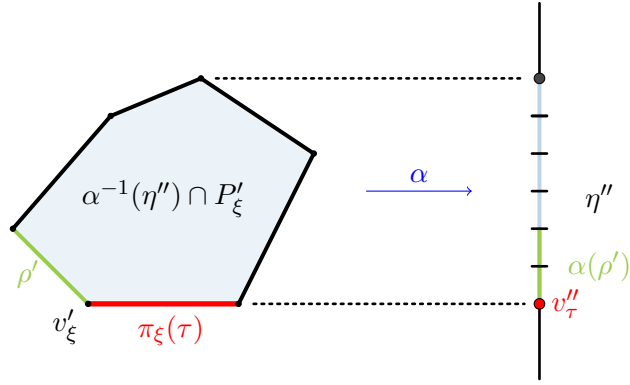
\begin{center}
\definecolor{uuuuuu}{rgb}{0.26666666666666666,0.26666666666666666,0.26666666666666666}
\definecolor{redcolor}{rgb}{1,0,0}
\definecolor{greencolor}{rgb}{0.6,0.8,0.3}
\definecolor{bcduew}{rgb}{0.7372549019607844,0.8313725490196079,0.9019607843137255}
\begin{tikzpicture}[line cap=round,line join=round,x=1cm,y=1cm]
\fill[line width=2pt,color=bcduew,fill=bcduew,fill opacity=0.28] (-3,0) -- (-1,0) -- (0,2) -- (-1.5,3) -- (-2.7,2.5) -- (-4,1) -- cycle;
\draw [line width=2pt,color=redcolor] (-3,0)-- (-1,0);
\draw [line width=1.5pt] (-1,0)-- (0,2);
\draw [line width=1.5pt] (0,2)-- (-1.5,3);
\draw [line width=1.5pt] (-1.5,3)-- (-2.7,2.5);
\draw [line width=1.5pt] (-2.7,2.5)-- (-4,1);
\draw [line width=1.5pt,color=greencolor] (-4,1)-- (-3,0);
\draw [line width=1pt] (3,4)-- (3,-1);
\draw [line width=1.5pt,color=bcduew] (3,3)-- (3,0);
\draw [line width=1.5pt,color=greencolor] (3,1)-- (3,0);
\draw [line width=1pt] (2.9,2.5)-- (3.1,2.5);
\draw [line width=1pt] (2.9,2.5)-- (3.1,2.5);
\draw [line width=1pt] (2.9,2)-- (3.1,2);
\draw [line width=1pt] (2.9,1.5)-- (3.1,1.5);
\draw [line width=1pt] (2.9,1)-- (3.1,1);
\draw [line width=1pt] (2.9,0.5)-- (3.1,0.5);
\draw [line width=1pt, dotted] (-0.8,0)-- (2.8,0);
\draw [line width=1pt,dotted] (-1.3,3)-- (2.8,3);
\begin{scriptsize}
\draw [fill=black] (-3,0) circle (1pt);
\draw [fill=black] (-1,0) circle (1pt);
\draw [fill=black] (0,2) circle (1pt);
\draw [fill=black] (-1.5,3) circle (1pt);
\draw [fill=black] (-2.7,2.5) circle (1pt);
\draw [fill=black] (-4,1) circle (1pt);
\draw [fill=uuuuuu] (3,3) circle (2pt);
\draw [fill=redcolor] (3,0) circle (2pt);
\end{scriptsize}
\draw[color=black] (3.35,3) node {}; 
\draw[color=black] (-3.3,-0.3) node {$v'_{\xi}$};
\draw[color=red] (-1.9,-0.35) node {$\pi_{\xi}(\tau)$};
\draw[color=greencolor] (-3.8,0.4) node {$\rho'$};
\draw[color=greencolor] (3.8,0.5) node {$\alpha(\rho')$};
\draw[color=black] (-2,1.4) node {$\alpha^{-1}(\eta'')\cap P'_{\xi}$};
\draw[color=black] (3.8,1.4) node {$\eta''$};
\draw[color=blue] (1.4,1.7) node {$\alpha$};
\draw[color=redcolor] (3.35,0) node {$v''_{\tau}$};
\draw[->,blue] (0.7,1.5) -- (2.1,1.5);
\end{tikzpicture}
\end{center}
We have that $|\rho'|_{M'_{\rho'}}\le |\alpha(\rho')|_{M''_{\eta''}}\le  |\eta''|_{M''_{\eta''}} =s(P''_\tau;v''_\tau)$, the first inequality follows the definition of $M''_{\eta''}$ and the second inequality follows from the fact that $\alpha{(\rho')}$ is a subset of $\eta''$.
Hence  $s(P'_\xi;v'_\xi)\le |\rho'|_{M_{\rho'}}\leq s(P''_\tau;v''_\tau)$. 
This inequality implies that $\varepsilon(X,D;p)\leq s(P''_\tau;v''_\tau)$ by Proposition~\ref{ito's prop3.12}.

From the equivalent definition of the Seshadri constant in \cite[Proposition 5.1.5]{PAG}, we have that  $\varepsilon(X,D;p) \leq \varepsilon(X_\tau,D_\tau;p)$. 
Since $O_\xi$ has codimension one in $X_\tau$, then $p$ is a smooth point in the toric variety $X_\tau$. Lemma \ref{generic-point} implies that  
$\varepsilon(X_\tau,D_\tau;p)\leq \varepsilon(X_\tau,D_\tau;q)$. We thus have $\varepsilon(X,D;p) \leq \varepsilon(X_\tau,D_\tau;q)$. By Proposition~\ref{ito's prop3.12}, we have 
\[\varepsilon(X,D;p)\leq \min\{\varepsilon(X_\tau,D_\tau;q),s(P''_\tau;v''_\tau)\}=\varepsilon(X,D;q). \qedhere
\]
\end{proof}


Note that the Seshadri constant function is constant on each $T$-orbit of the projective toric variety $X$. We deduce that the set $\{p\in X~|~\varepsilon(X,D;p)=\varepsilon(X,D)\}$ is $T$-invariant and closed, and in particular it contains a $T$-invariant point. 

\begin{corollary} \label{corollary.seshadri.invariant.points}
Let $D$ be an ample Cartier $T$-divisor on a projective toric variety $X$, Then $\varepsilon(X,D)$ is equal to $\min\{ \varepsilon(X,D;p) \, | \,  p \in X \textnormal{ is $T$-invariant} \}$.
In particular, $\varepsilon(X,D)$ is an integer. 
\end{corollary}



\bibliographystyle{alpha}
\bibliography{JetAmplenessFujita}

\newcommand{\etalchar}[1]{$^{#1}$}
\begin{thebibliography}{BRH{\etalchar{+}}09}

\bibitem[BDRS00]{BRSz00-Jets-K3}
Thomas Bauer, Sandra Di~Rocco, and Tomasz Szemberg.
\newblock Generation of jets on k3 surfaces.
\newblock {\em Journal of Pure and Applied Algebra}, 146(1):17--27, 2000.

\bibitem[BFS89]{BFS89}
Mauro Beltrametti, Peter Francia, and Andrew Sommese.
\newblock On {R}eider’s method and higher order embeddings.
\newblock {\em Duke Math. J.}, 58(2):425--439, 04 1989.

\bibitem[BRH{\etalchar{+}}09]{bauer2009primer}
Thomas Bauer, S~Di Rocco, Brian Harbourne, Micha{\l} Kapustka, Andreas Knutsen,
  Wioletta Syzdek, and Tomasz Szemberg.
\newblock A primer on seshadri constants.
\newblock {\em Contemporary Mathematics}, 496:33, 2009.

\bibitem[BS90]{BS90}
Mauro Beltrametti and Andrew Sommese.
\newblock On k-spannedness for projective surfaces.
\newblock In Andrew Sommese, Aldo Biancofiore, and Elvira~Laura Livorni,
  editors, {\em Algebraic Geometry}, pages 24--51, Berlin, Heidelberg, 1990.
  Springer Berlin Heidelberg.

\bibitem[BS91]{BS91}
Mauro Beltrametti and Andrew Sommese.
\newblock Zero cycles and k-th order embeddings of smooth projective surfaces.
\newblock In {\em Problems in the theory of surfaces and their classification,
  Symposia Math}, volume~32, pages 33--48, 1991.

\bibitem[BS93]{BS93}
Mauro Beltrametti and Andrew Sommese.
\newblock On k-jet ampleness.
\newblock In {\em Complex analysis and geometry}, pages 355--376. Springer,
  1993.

\bibitem[BS97a]{BaSz97}
Thomas Bauer and Tomasz Szemberg.
\newblock Higher order embeddings of abelian varieties.
\newblock {\em Mathematische Zeitschrift}, 224(3):449--455, 1997.

\bibitem[BS97b]{BaSz97-Surfaces}
Thomas Bauer and Tomasz Szemberg.
\newblock Primitive higher order embeddings of abelian surfaces.
\newblock {\em Transactions of the American Mathematical Society},
  349(4):1675--1683, 1997.

\bibitem[BS00]{Calabi-Yau-Threefolds-2000}
Mauro Beltrametti and Tomasz Szemberg.
\newblock On higher order embeddings of {C}alabi-{Y}au threefolds.
\newblock {\em Archiv der Mathematik}, 74(3):221--225, Mar 2000.

\bibitem[CI14]{Chintapalli2014}
Seshadri Chintapalli and Jaya N.~N. Iyer.
\newblock Embedding theorems on hyperelliptic varieties.
\newblock {\em Geometriae Dedicata}, 171(1):249--264, Aug 2014.

\bibitem[DR99]{DiRoccoJets99}
Sandra Di~Rocco.
\newblock Generation of k-jets on toric varieties.
\newblock {\em Mathematische Zeitschrift}, 231(1):169--188, 1999.

\bibitem[EKL95]{EKL95}
Lawrence Ein, Oliver K\"{u}chle, and Robert Lazarsfeld.
\newblock Local positivity of ample line bundles.
\newblock {\em J. Differential Geom.}, 42(2):193--219, 1995.

\bibitem[EL93]{EL93}
Lawrence Ein and Robert Lazarsfeld.
\newblock Global generation of pluricanonical and adjoint linear series on
  smooth projective threefolds.
\newblock {\em J. Amer. Math. Soc.}, 6(4):875--903, 1993.

\bibitem[EW91]{EW91}
G\"{u}nter Ewald and Uwe Wessels.
\newblock On the ampleness of invertible sheaves in complete projective toric
  varieties.
\newblock {\em Results Math.}, 19(3-4):275--278, 1991.

\bibitem[Far16]{Farnik2016}
{\L}ucja Farnik.
\newblock On k-jet ampleness of line bundles on hyperelliptic surfaces.
\newblock {\em Mediterranean Journal of Mathematics}, 13(6):4783--4804, Dec
  2016.

\bibitem[Fuj88]{Fuj88}
Takao Fujita.
\newblock Contribution to birational geometry of algebraic varieties: Open
  problems.
\newblock In {\em Proceedings of the 23rd International Symposium, Division of
  Mathematics, Taniguchi Foundation, Katata, Japan}, 1988.

\bibitem[Fuj03]{Fuj03}
Osamu Fujino.
\newblock Notes on toric varieties from {M}ori theoretic viewpoint.
\newblock {\em Tohoku Math. J. (2)}, 55(4):551--564, 2003.

\bibitem[Ful93]{Fulton}
William Fulton.
\newblock {\em Introduction to toric varieties}, volume 131 of {\em Annals of
  Mathematics Studies}.
\newblock Princeton University Press, Princeton, NJ, 1993.
\newblock The William H. Roever Lectures in Geometry.

\bibitem[Hum12]{Linear.Algebraic.Groups.Humphreys}
James Humphreys.
\newblock {\em Linear algebraic groups}, volume~21.
\newblock Springer Science \& Business Media, 2012.

\bibitem[Ito14]{Ito12}
Atsushi Ito.
\newblock Seshadri constants via toric degenerations.
\newblock {\em J. Reine Angew. Math.}, 695:151--174, 2014.

\bibitem[Kaw97]{Kaw97}
Yujiro Kawamata.
\newblock On {F}ujita's freeness conjecture for {$3$}-folds and {$4$}-folds.
\newblock {\em Math. Ann.}, 308(3):491--505, 1997.

\bibitem[Lat96]{laterveer1996}
Robert Laterveer.
\newblock Linear systems on toric varieties.
\newblock {\em Tohoku Math. J. (2)}, 48(3):451--458, 1996.

\bibitem[Laz04]{PAG}
Robert Lazarsfeld.
\newblock {\em Positivity in algebraic geometry {I}}, volume~48.
\newblock Springer Science \& Business Media, 2004.

\bibitem[Mus02]{Mustata02}
Mircea Musta\c{t}\u{a}.
\newblock Vanishing theorems on toric varieties.
\newblock {\em Tohoku Math. J. (2)}, 54(3):451--470, 2002.

\bibitem[Pay06]{Pay06}
Sam Payne.
\newblock Fujita's very ampleness conjecture for singular toric varieties.
\newblock {\em Tohoku Math. J. (2)}, 58(3):447--459, 2006.

\bibitem[RS04]{RamsSz04-Jets-K3}
Slawomir Rams and Tomasz Szemberg.
\newblock Simultaneous generation of jets on {K}3 surfaces.
\newblock {\em Archiv der Mathematik}, 83(4):353--359, 2004.

\bibitem[RT12]{RT12}
Michele Rossi and Lea Terracini.
\newblock Linear algebra and toric data of weighted projective spaces.
\newblock {\em Rend. Semin. Mat. Univ. Politec. Torino}, 70(4):469--495, 2012.
\newblock Extended version: arXiv:1112.1677.

\bibitem[YZ20]{YZ15}
Fei Ye and Zhixian Zhu.
\newblock On {F}ujita's freeness conjecture in dimension 5.
\newblock {\em Advances in Mathematics}, 371:Article 107210, 2020.

\end{thebibliography}

\end{document}